\documentclass[12pt]{amsart}
\usepackage[colorlinks=true,pagebackref,hyperindex,citecolor=Green,linkcolor=Blue]{hyperref}
\usepackage[top=1in, bottom=1in, left=1.1in, right=1.1in]{geometry}
\usepackage{amsmath}
\usepackage{amsfonts}
\usepackage{amssymb}
\usepackage{amsthm}
\usepackage[all]{xy}
\usepackage{mathrsfs}
\usepackage{enumitem} 
\usepackage[T1]{fontenc}
\usepackage{color}
\usepackage{stmaryrd}
\usepackage{bigints}

\usepackage{graphicx} 
\usepackage{tikz-cd} 
\usepackage{tikz} 

\makeatletter
\def\namedlabel#1#2{\begingroup
#2%
\def\@currentlabel{#2}%
\phantomsection\label{#1}\endgroup
}
\makeatother

\newtheorem{theorem}{Theorem}[section]

\newtheorem{corollary}[theorem]{Corollary}

\newtheorem{lemma}[theorem]{Lemma}

\newtheorem{proposition}[theorem]{Proposition}

\newtheorem{theoremx}{Theorem}

\newtheorem{propositionx}[theoremx]{Proposition}

\theoremstyle{definition} 
\newtheorem{definition}[theorem]{Definition}

\newtheorem{example}[theorem]{Example}
\newtheorem{remark}[theorem]{Remark}
\numberwithin{equation}{subsection}

\definecolor{blue-violet}{rgb}{0.54, 0.17, 0.89}
\definecolor{Blue}{rgb}{0.01, 0.28, 1.0}
\definecolor{gGreen}{rgb}{0.2, 0.8, 0.2}
\definecolor{Green}{rgb}{0.04, 0.85, 0.32}

\makeatletter
\def\@tocline#1#2#3#4#5#6#7{\relax
  \ifnum #1>\c@tocdepth 
  \else
    \par \addpenalty\@secpenalty\addvspace{#2}%
    \begingroup \hyphenpenalty\@M
    \@ifempty{#4}{%
      \@tempdima\csname r@tocindent\number#1\endcsname\relax
    }{%
      \@tempdima#4\relax
    }%
    \parindent\z@ \leftskip#3\relax \advance\leftskip\@tempdima\relax
    \rightskip\@pnumwidth plus4em \parfillskip-\@pnumwidth
    #5\leavevmode\hskip-\@tempdima
      \ifcase #1
       \or\or \hskip 1.9em \or \hskip 2em \else \hskip 3em \fi%
      #6\nobreak\relax
    \dotfill\hbox to\@pnumwidth{\@tocpagenum{#7}}\par
    \nobreak
    \endgroup
  \fi}
\makeatother

\newcommand{\NN}{\mathbb{N}}
\newcommand{\RR}{\mathbb{R}}
\newcommand{\ZZ}{\mathbb{Z}}
\newcommand{\QQ}{\mathbb{Q}}
\newcommand{\CC}{\mathbb{C}}
\newcommand{\KK}{\mathbb{K}}

\newcommand{\cM}{\mathscr{M}}
\newcommand{\cN}{\mathscr{N}}

\newcommand{\cO}{{\mathcal O}}

\newcommand{\cJ}{{\mathcal J}}

\newcommand{\cB}{\mathcal{B}}

\newcommand{\Hom}{\operatorname{Hom}}

\newcommand{\Frac}{\operatorname{Frac}}

\newcommand{\lct}{\operatorname{lct}}

\newcommand{\gr}{\operatorname{gr}}

\newcommand{\Der}{\operatorname{Der}}

\newcommand{\ord}{\operatorname{ord}}

\newcommand{\bfg}{{\bf \frac{f^s}{g^s}}}
\newcommand{\bfga}{{\bf \frac{f^s}{g^{s+\alpha}}}}
\newcommand{\bfgt}{{\bf \frac{\tilde{f}^s}{\tilde{g}^{s+\alpha}}}}

\newcommand{\ga}{{\bf \frac{1}{g^{\alpha}}}}

\begin{document}

\title[Bernstein-Sato polynomial for meromorphic functions]{Bernstein-Sato polynomial and related invariants for meromorphic functions}

\dedicatory{In memory of Professor Roberto Callejas-Bedregal}

\author[\`Alvarez Montaner]{Josep \`Alvarez Montaner{$^1$}}
\address{Departament de Matem\`atiques  and  Institut de Matem\`atiques de la UPC-BarcelonaTech (IMTech)\\  Universitat Polit\`ecnica de Catalunya \\ Av.~Diagonal 647, Barcelona 08028; and Centre de Recerca Matem\`atica (CRM), Bellaterra, Barcelona 08193.} 
\thanks{{$^1$}Partially supported by projects reference  PID2019-103849GB-I00 and PID2023-146936NB-I00 funded by the Spanish State Agency MICIU/AEI/10.13039/501100011033. He is also supported by and AGAUR grant 2021 SGR 00603 and  Spanish State Research Agency, through the Severo Ochoa and Mar\'ia de Maeztu Program for Centers and Units of Excellence in R$\&$D (project CEX2020-001084-M).}
\email{josep.alvarez@upc.edu}

\author[Gonz\'alez Villa]{Manuel Gonz\'alez Villa$^{2}$}
\address{
Centro de investigaci\'on en 
Matem\'aticas\\
Apartado Postal
402 \\
C.P. 36000, Guanajuato, GTO, M\'exico.
}
\thanks{$^{2}$Partially supported by CONACyT Grant No. 320393 and by Grant PID2020-114750GB-C33 funded by
 MCIN/AEI/10.13039/501100011033.}
\email{manuel.gonzalez@cimat.mx}

\author[Le\'on-Cardenal]{Edwin Le\'on-Cardenal$^{3}$}
\address{Departamento de Matem\'aticas, IUMA, Universidad de Zaragoza\\ 
	C. Pedro Cerbuna 12, 50009, Zaragoza, Spain.}
\email{eleon@unizar.es}
\address{CONACYT--Centro de investigaci\'on en 
	Matem\'aticas\\
	Unidad Zacatecas, Av. Lasec Andador Galileo Galilei, Manzana 3 Lote 7\\ 
	C.P.  98160,
	Zacatecas, ZAC, M\'exico.
}
\email{edwin.leon@cimat.mx}
\thanks{$^{3}$ Partially supported by CONACyT Grant No. 286445, Grant PID2020-114750GB-C31 funded by MCIN/AEI/10.13039/501100011033, and the
	Departamento de Ciencia, Universidad y Sociedad del Conocimiento of the
	Gobierno de Arag\'on (Grupo de referencia “\'Algebra y Geometr\'ia”) $E22\_20R$. 
	Funded by European Union-Next Generation EU under Mar\'ia Zambrano program.}

\author[N\'u\~nez-Betancourt]{Luis N\'u\~nez-Betancourt$^{4}$}
\address{
Centro de investigaci\'on en 
Matem\'aticas\\
Apartado Postal
402 \\
C.P. 36000, Guanajuato, GTO, M\'exico.
}
\thanks{$^{4}$ Partially supported by  CONACyT Grant 284598 and C\'atedras Marcos Moshinsky.}
\email{luisnub@cimat.mx}




\begin{abstract}
We develop a theory of Bernstein-Sato polynomials for meromorphic functions. As a first application we study the poles of Archimedian local zeta functions for meromorphic germs. We also present a  theory of multiplier ideals for meromorphic functions from the analytic and algebraic point of view. It is also shown that the jumping numbers of these multiplier ideals are related with the roots of the Bernstein-Sato polynomials.
\end{abstract}

\maketitle

\setcounter{tocdepth}{1}
\tableofcontents

\section{Introduction} 

The theory of Bernstein-Sato polynomials was introduced independently by Bernstein \cite{Ber72} to study the analytic continuation of the Archimedean zeta function and Sato \cite{SatoPoly} in the context of prehomogeneous vector spaces.  Rapidly, Bernstein-Sato polynomials became an indispensable tool to study invariants of singularities of holomorphic or regular functions. In particular, the roots of the Bernstein-Sato polynomials are related to many other invariants such as poles of zeta functions, eigenvalues of the monodromy of the Milnor fiber, jumping numbers of multiplier ideals, spectral numbers and F-thresholds, among others.

The aim of this paper is to develop a theory of Bernstein-Sato polynomials for the case of meromorphic functions, and to relate it to other invariants of such functions.   
This is inspired in previous work that has been done for singularities of meromorphic functions.
For instance, Arnold proposed the study of singularities of meromorphic functions and began its classification 
\cite{MR1683686}. He remarked that even simple examples provide new interesting phenomena. Gussein-Zade, Luengo, and Melle \cite{MR1647824, MR1725939, MR1734347, MR1849313} pursued  a systematic research of the topology and monodromy of germs of complex meromorphic functions (see also  works  by Tibar and Siersma 
\cite{MR1900789, MR2087815}, Bodin, Pichon, and  Seade  \cite{MR2318645, MR2430986,MR2545254}). Raibaut  \cite{MR2929060,MR3068616} studied motivic zeta functions and motivic Milnor fibers for meromorphic germs, also considered by Libgober, Maxim, and the second author  \cite{GVLMad, GVLMcm}. Z\'u\~niga-Galindo and the third author treated $p$-adic zeta functions for Laurent polynomials \cite{LeZu}. Lemahieu and the second author treated the case of topological zeta functions of meromorphic germs, and proved a generalization of the monodromy conjecture in the two variable case \cite{GVLLem}.   Veys and Z\'u\~niga-Galindo  \cite{Ve-Zu} investigated meromorphic continuation and poles of zeta functions and oscillatory integrals for meromorphic functions defined over local fields of characteristic zero (see also \cite{Leon,MR3779568, bocardogaspar2020poles}).
Nguyen and Takeuchi  \cite{nguyen2019meromorphic} introduced meromorphic nearby cycle functors, and studied some of its applications to the monodromy of meromorphic germs. 

In this paper, we define several types of Bernstein-Sato polynomials for meromorphic functions.  The more general version is a family of Bernstein-Sato polynomials $b_{f/g}^\alpha(s)$ indexed by $\alpha\in \mathbb{R}_{\geq 0}$; however, it only exists for all but finitely many values of $\alpha$ (see Definition \ref{Def:Module_M_alpha} for the module where the equation occurs). The second version mimics the classical case and is used in Theorem \ref{Thm: JN_rootsBS} to study multiplier ideals. 

\begin{theoremx} [{Theorem \ref{Thm:GeneralizedFuncEq}}]\label{MainThmA}
Let $R$ be either a polynomial ring $R=\KK[x_1,\dots,x_n]$ or a ring of holomorphic functions $R=\CC\{x_1,\dots,x_n\}$. Let $f,g\in R$ be nonzero elements. Then: 
\begin{enumerate}
\item  For almost all $\alpha\in\RR_{\geq 0}$,  
there exists $b(s)\in \KK[s]\setminus \{0\}$ and $\delta(s)\in D_{R|\KK}[s]$ such that
$$
\delta(s)  f \bfga= b(s)\bfga.
$$
\item  There exists $b(s)\in \KK[s]\setminus \{0\}$ and $\delta(s)\in D_{R|\KK}[s]$ such that
$$\delta(s)\frac{f}{g} \bfg= b(s)\bfg.$$
\end{enumerate}

\end{theoremx}
We point out that our construction in Theorem \ref{MainThmA} uses critically the work of Sabbah \cite{Sabbah}. 
The parameter $\alpha$ in the first part of Theorem \ref{MainThmA} is needed for our applications to zeta functions (see  Theorem \ref{MainThmZeta}), which can be obtained even if $b_{f/g}^\alpha(s)=0$ for some $\alpha$. 

We also provide a version of Kashiwara's proof of the rationality of the roots of the Bernstein-Sato polynomial $b_{f/g}^\alpha(s)$ \cite{Kashiwara} and the refinement given by Lichtin \cite{Lic89}.

\begin{theoremx} [{Theorem  \ref{Thm:Kashiwara}, Corollary  \ref{Cor:Kashiwara}}]
Let $f,g\in \CC\{ x_1,\dots , x_n\}$ be  nonzero holomorphic functions and let $\alpha\in\RR_{\geq 0}$ such that $b_{f/g}^\alpha\neq 0$.
The roots of $b_{f/g}^\alpha(s)$ are negative rational numbers. Moreover, they are contained in the set
$$\Bigg\{ \frac{k_i+1+ \ell }{N_{f/g, i}}  \;  ; \;  \ell \geq 0, \;  i\in I_0\Bigg\},$$ where the integers $N_{f/g, i}$ and $k_i$ are extracted from the numerical data of a log resolution of $f/g$.
\end{theoremx}

As an application, we study meromorphic continuation of  Archimedean local zeta functions of meromorphic germs, extending the ideas of Bernstein  \cite{Ber72}  in the classical case. Our approach is complementary to the one done by  Veys and  Z\'u\~niga-Galindo  \cite{Ve-Zu}, where they used resolution of singularities following the ideas of Bernstein and Gel'fand \cite{Ber-Gel} and Atiyah  \cite{Ati}.

 \begin{theoremx} [{Theorem  \ref{Thm:MerContZetaF}}]\label{MainThmZeta}
	Let $f$ and $g$ be nonzero holomorphic functions and let 
	\[\alpha=\max_{k\geq 1}\left\{\alpha^\prime=\frac{\lct_0 g}{k} \ ; \ b^{\alpha^\prime}_{f/g}(s)\neq 0\right\}\quad\text{and}\quad \beta=\max_{k\geq 1}\left\{-\beta^\prime=\frac{\lct_0 f}{k} \ ; \ b^{\beta^\prime}_{g/f}(s)\neq 0\right\}.
	\]
	The local zeta function $Z_{\phi}(s,f/g)$ has a meromorphic continuation to the whole complex plane $\CC$, and  its poles  are contained in the set 
	\[\Bigl\{\zeta-\ell\alpha \ ;\ \ell\in\ZZ_{\geq 0}\Bigr\}_{\zeta \text{ root of } b^\alpha_{f/g}(s)}\bigcup \Bigl\{\ell\beta-\xi \ ;\ \ell\in\ZZ_{\geq 0}\Bigr\}_{\xi \text{ root of } b^\beta_{g/f}(s)}.\]
	In particular, the poles of $Z_{\phi}(s,f/g)$ are rational numbers.
\end{theoremx}

We also develop a theory of multiplier ideals for meromorphic functions from the analytic and algebraic point of view.

\begin{theoremx} 
 Let $f$ and $g$ be nonzero holomorphic functions. Then:
\begin{itemize}

\item  {\rm (Theorem  \ref{Thm: ratJN})}
The set of jumping numbers of $f/g$  is a set of rational numbers with no accumulation points.

\item  {\rm (Theorem \ref{Thm: JN_rootsBS})}
Let $\lambda$ be a jumping number of $f/g$ such that $\lambda \in (1-\lct_0(g), 1]$. Then, $- \lambda$ is a root of the Bernstein-Sato polynomial of $f/g$.

\end{itemize}	

\end{theoremx}

There is also a weaker version of Skoda's Theorem in this context.

\begin{propositionx} 
 Let $f$ and $g$ be nonzero holomorphic functions. Then:
	
\begin{itemize}
\item {\rm (Proposition  \ref{PropSkoda})} $
\cJ ( ( \frac{f}{g})^{\lambda+\ell})
=
( \frac{f^\ell}{g^\ell}( \cJ  ( \frac{f}{g})^\lambda))\bigcap R
$
for every $\ell\in\NN.$
In particular,
if  $\lambda+1$ is a jumping number, then $\lambda$ is a jumping number.

\item {\rm (Lemma \ref{Lem: JN_integers})}  For every $\lambda \in \RR_{> 0}$, we have  $\cJ(f^\lambda) \subseteq  \cJ((\frac{f}{g})^\lambda)$. In addition, we have $\cJ((\frac{f}{g})^n)=(f^n)$  for every  $n \in \mathbb{Z}_{>0}$.
\end{itemize}	

\end{propositionx}

Finally we mention that Takeuchi has developed independently a theory of Bernstein-Sato polynomials for meromorphic functions \cite{takeuchi2023}. Both approaches differ slightly but complement each other. Takeuchi is interested in the relation between the roots of the Bernstein-Sato polynomial and the eigenvalues of the Milnor monodromies. Meanwhile, we pay attention to the meromorphic continuation of the Archimedean local zeta function and multiplier ideals. We also give a new proof for the existence of Takeuchi's Bernstein-Sato polynomial in Theorem \ref{Thm:Takeuchi}, which is based on holonomicity and not on $V$-filtrations.


\section{Log-resolution of meromorphic functions} \label{Sec:log_resolution}
Let  $f,g:(\CC^{n},0) \longrightarrow (\CC,0)$ be two germs of holomorphic functions and we refer to $f/g: (\CC^{n},0) \longrightarrow (\CC,0)$ as a germ of a meromorphic function. Taking local coordinates we assume $f,g \in R= \CC\{ x_1,\dots , x_n\}$. Recall that $f/g$ and $f'/g'$ define the same germ if there exist a unit $u \in \mathcal{O}_{\CC^{n},0}$ such that  $f= u f'$ and $g=u g'$.  

\begin{definition}
Let $X$ be a $n$-dimensional smooth analytic manifold, $U$ a neighborhood of $0 \in \mathbb{C}^{n}$ and $\pi : X \to U$ a proper analytic map. We say that $\pi$ is  log resolution of the meromorphic germ   $f / g$ if:
\begin{itemize}
\item $\pi$  is a log resolution of the hypersurface $H = \{f_{|_U} = 0\} \cup \{g_{|_U} = 0\}$, i.e. $\pi$ is an isomorphism outside a proper analytic subspace in $U$;

\item there is a 
normal crossing divisor $F$ on $X$ such that $\pi^{-1}(H) = \mathcal{O}_X(-F)$;

\item the lifting $\tilde{f} / \tilde{g} = (f/g)  \circ \pi = \frac{f \circ \pi}{g \circ \pi}$ defines a holomorphic map $\tilde{f} / \tilde{g} : X \to \mathbb{P}^1$.

\end{itemize}

\end{definition}

One can obtain a log resolution $\pi$ of $f/g$   from a log resolution
$\pi'$  of  $f \cdot g$ by blowing up along the intersections of irreducible components of
$\pi'^{-1}\{f \cdot g = 0\}$ until the irreducible components of the strict transform in $\pi^{-1}\{f = 0\}$
and $\pi^{-1}\{g = 0\}$ are separated by a dicritical component, i.e. an exceptional divisor $E$
of $\pi$ for which  $(\tilde{f}/ \tilde{g})_{|_E} : E \to \mathbb{P}^1$ is a surjective map.

Let  $\{E_i\}_{i \in I}$ be the irreducible components of $F$. We denote the relative canonical divisor, defined by the Jacobian  determinant of $\pi$, as the divisor 
$$K_\pi = \sum k_i E_i .$$
We also denote the total transforms of $f$ and $g$  as
\[ \tilde{f}=\pi^* f= \sum N_{f, i} E_i, \quad \tilde{g}= \pi^* g= \sum N_{g, i} E_i. \]
Moreover, we  define $$  N_{f/g,i}=N_{f,i}-N_{g,i}.$$
Notice that $E_i$ is a dicritical component if and only if $N_{f/g,i}=0$. 
It is then natural to  associate  with the meromorphic germ $f/g$ the  divisor $$\tilde{F} = \sum_{i \in I} N_{f/g,i} E_i.$$  For $i \in I$ we write 
$i \in I_0$ if $N_{f,i} > N_{g,i}$, $i\in I_\infty$ if $N_{g,i} > N_{f,i}$, and $i \in I_d$ if $N_{f_i}=N_{g,i}$. 
We have the decomposition $\tilde{F} = \tilde{F}_0 + \tilde{F}_\infty+ \tilde{F}_d$ where  $$\tilde{F}_0=\sum_{i \in I_0} N_{f/g,i}E_i,  \hskip 5mm \tilde{F}_\infty=\sum_{i \in I_\infty} N_{f/g,i} E_i, \hskip 5mm \textrm{and} \hskip 5mm \tilde{F}_d=\sum_{i \in I_d} N_{f/g,i} E_i,$$ with $\tilde{F}_d=0$ by definition.

\begin{remark}
Around a given point $p \in X$ we may consider a local system of coordinates $(z_1,\dots, z_{n})$ such that the components $E_i$ containing $p$ are given by the equations $z_i=0$. Assume that $p\in E_i$, for $i=1,\dots, m$, then we have the local equations:
\[ \tilde{f} = u z_1^{N_{f,1}} \cdots z_m^{N_{f,m}},\quad\text{and}\quad \tilde{g}= v z_1^{N_{g,1}} \cdots z_m^{N_g,m},\]
where $u,v \in \cO_{X,p}$ are units. Moreover, the local equations of the relative canonical divisor are
\[\tilde{\omega}=\pi^\ast(dx_1 \wedge \cdots\wedge dx_{n})= w (z_1^{k_1} \cdots  z_m^{k_m}) \, dz_1 \wedge \cdots \wedge dz_{n},\] 
where $w \in \cO_{X,p}$ is a unit as well. Furthermore, applying some extra blow-ups if necessary, we may assume that either $\tilde{f}$ divides $\tilde{g}$ or the other way around, so we have either
$$\frac{\tilde{f}}{\tilde{g}} = \frac{u}{v} \;   z_1^{N_{f/g,1}} \cdots z_m^{N_{f/g,m}} \hskip 5mm \textrm{or} \hskip 5mm  \frac{\tilde{f}}{\tilde{g}} = \frac{u}{v}  \;  \frac{1}{z_1^{|N_{f/g,1}|} \cdots z_m^{|N_{f/g,m}|}} .$$
\end{remark}

\section{Meromorphic Bernstein-Sato polynomial}\label{Sec3}

In this section we extend the classical Bernstein-Sato theory to the case of meromorphic functions.  Indeed we treat different types of Bernstein-Sato polynomials with a view towards relations with other invariants of singularities. First we recall the basics on the theory of rings of differential operators as introduced  by Grothendieck \cite[\S 16.8]{EGA}.
Next, we introduce a  Bernstein-Sato polynomial of a meromorphic function that mimics the classical case. This construction is used to study the jumping numbers of the meromorphic multiplier ideals we introduce in Section \ref{Sec5}. Then we provide another Bernstein-Sato polynomial of a meromorphic function that depends on a parameter $\alpha \in \RR_{\geq 0}$. This construction is used to study the analytic continuation of the Archimedian local zeta function of a meromorphic function in Section \ref{Sec4}. We also recover, for the case of polynomial rings, the notion of Bernstein-Sato polynomials for meromorphic functions introduced by Takeuchi \cite{takeuchi2023}.

\subsection{Rings of differential operators and its modules}
\begin{definition}
	Let $R$ be a Noetherian ring containing a field of characteristic zero $\KK$. The ring of $\KK$-linear differential
	operators of $R$ is the subring $D_{R|\KK}\subseteq \Hom_{\KK}(R,R)$,
	whose elements are defined inductively as follows.
	\begin{itemize}
		\item   Differential
		operators of order zero are defined by the multiplication by
		elements of $R$, and so, $D^{0}_{R|\KK}\cong R.$ 
		\item An element 
		$\delta\in \Hom_\ZZ(R,R)$ is an operator of order less than or equal
		to $n$ if $[\delta,r]=\delta r-r\delta$ is an operator of order less
		than or equal to $n-1.$ 
	\end{itemize}
	We denote by $D^n_{R|\KK}$ the set of all differential operators of order  less than or equal
	to $n$.	We  have a filtration, $D^{\bullet}_{R|\KK}$, given by $D^{0}_{R|\KK} \subseteq
	D^{1}_{R|\KK} \subseteq \cdots$ such that $D^m_{R|\KK}D^n_{R|\KK}\subseteq D^{m+n}_{R|\KK}$. 
	The ring of	differential operators is defined as
	\[D_{R|\KK}=\bigcup_{n\in\NN}D^{n}_{R|\KK}.
	\]
\end{definition}

Now we collect some properties of the class of holonomic $D$-modules in the particular case of polynomial rings. The interested reader may consult \cite{Cou} and \cite{Castro} for further insight.

\begin{definition}
	Let $R=\KK[x_1,\ldots, x_d]$.
	We define the Bernstein filtration of $R$, $\cB^\bullet_{R|\KK}$ by
	$$
	\cB^i_{R|\KK}=\KK\{x^\alpha \delta^\beta \; ;\; |\alpha|+|\beta|\leq i\}.
	$$
	We have 
	\begin{itemize}
		\item $\dim_\KK  \cB^i_{R|\KK}=\binom{n+i}{i}<\infty$,
		\item $D_{R|\KK}=\bigcup_{i\in\NN} \cB^i_{R|\KK}$, and 
		\item $ \cB^n_{R|\KK} \cB^j_{R|\KK} =\cB^j_{R|\KK}$.
	\end{itemize}
\end{definition}
and $\gr_{\cB^\bullet_{R|\KK}}(D_{R|\KK})$ is a commutative ring isomorphic to 
$\KK[x_1,\ldots, x_d,y_1,\ldots,y_d]$.

\begin{definition}
	Given a left (right) $D_{R|\KK}$-module, $M$, we say that a filtration $\Gamma^\bullet$ of $\KK$-vector spaces is $\cB^\bullet_{R|\KK}$-compatible if
	\begin{itemize}
		\item $\dim_\KK  \Gamma^i<\infty$,
		\item $M=\bigcup_{i\in\NN} \Gamma^i$,
		\item $ \cB^i_{R|\KK} \Gamma^j\subseteq \Gamma^j$ ($ \Gamma^j \cB^i_{R|\KK}\subseteq \Gamma^j$).
	\end{itemize}
	We say that $\Gamma^\bullet$ is a good filtration if 
	$\gr_{\Gamma^\bullet}(M)$ is finitely generated as a  $\gr_{\cB^\bullet_{R|\KK}}(D_{R|\KK})$-module. 
\end{definition}

If $\Gamma^\bullet$ is a good filtration for $M$, we have that
$\dim_\KK \Gamma^n$ is a polynomial function on $n$ of degree equal to the Krull dimension
of $\gr_{\Gamma^\bullet}(M)$ by Hilbert-Samuel theory. 
This degree does not depend of the choice of such good filtration;  it is called the dimension of $M$ and denoted $\dim_{D_{R|\KK}} (M)$. 
We note that $M$ admits a good filtration if and only if it is a finitely generated  $D_{R|\KK}$-module \cite[Theorem 2.3]{Cou}.

\begin{theorem}[Bernstein's  Inequality]
	Let $M$ be a finitely generated  $D_{R|\KK}$-module.
	Then,
	$$
	d\leq  \dim_{D_{R|\KK}} (M)\leq 2d.
	$$
\end{theorem}

\begin{definition}
	A  $D_{R|\KK}$-module $M$, is holonomic if either $M=0$ or  $\dim_{D_{R|\KK}} (M)=d$. 
\end{definition}
We observe that a  $D_{R|\KK}$-module $M$, with a good filtration $\Gamma^\bullet$ is holonomic
if and only if there exists a polynomial $q$ in one variable of degree $d$, such that $\dim_\KK \Gamma^n\leq q(n)$.
We recall that every holonomic $D_{R|\KK}$-module has finite length. 

\subsection{A $D_{R|\KK}[s]$-module for meromorphic functions}
All the instances of the Bernstein-Sato polynomials for meromorphic functions take place in the following module.

\begin{definition}\label{Def:Module_M_alpha}
Let $R$ be either a polynomial ring $R=\KK[x_1,\dots,x_n]$ or a ring of holomorphic functions $R=\CC\{x_1,\dots,x_n\}$.
Let $f,g\in R$ be nonzero elements, take $\alpha\in\RR_{\geq 0}$ and
set $\cM^\alpha_{f/g}[s]= R_{fg}[s] \bfga$. This free $R_{fg}[s]$-module has an structure of $D_{R|\KK}[s]$-module given by
\[
\partial \cdot \frac{h(s)}{f^a g^b}\bfga =   
\partial \left( \frac{h(s)}{f^a g^b} \right)\bfga 
+ s\frac{\partial(f) h(s)}{f^{a+1}g^b}\bfga
- (s+\alpha)\frac{h(s) \partial(g)}{f^a g^{b+1}} \bfga,
\]
where $\partial$ is a $\KK$-linear derivative on $R$.
We denote  $D_{R(s)| \KK(s)}$ by $D(s)$ and $\cM^\alpha_{f/g}[s]\otimes_{\KK[s]}\KK(s)$ by $\cM^\alpha_{f/g}(s)$. Here $\KK(s)$ denotes the fraction field of $\KK[s]$. If $\alpha=0$, we just write $\cM_{f/g}[s]= R_{fg} [s]\bfg$ and  $\cM_{f/g}(s)= R_{fg} (s)\bfg$.
\end{definition}

\begin{lemma}\label{Lemma:Holonomicity}
	Let $R=\KK[x_1,\ldots,x_d]$ be a polynomial ring and let $f,g\in R$ be two nonzero elements. Then, $\cM^\alpha_{f/g}(s)$ is a holonomic $D(s)$-module.	
\end{lemma}
	\begin{proof}
		We fix $\alpha\in\RR_{\geq 0}$ and let $\theta=\deg(f)+\deg(g)$. 
		We set a filtration of finite dimensional $\CC$-vector spaces.
		\[\Gamma^n=\frac{1}{f^ng^n}\left.\left\{ h \bfga \, \right| \, \deg(h)\leq (\theta+1)n \right\}.
		\]
		Note that $x_i \Gamma^n\subseteq \Gamma^{n+1}$ and moreover $\partial_i \Gamma^n\subseteq \Gamma^{n+1}$, as we shall see next. Given $\frac{h}{f^ng^n} \bfga \in \Gamma^n$, we have that
		\begin{align*}
			\partial_i \frac{h}{f^ng^n} \bfga  &=
			\frac{f^ng^n \partial_i (h) - nhf^{n-1}g^{n-1} \partial_i(fg)}{f^{2n} g^{2n}}\bfga
			+
			\frac{shg\partial_i (f)-(s+\alpha)hf\partial_i(g)}{f^{n+1}g^{n+1}}  \bfga\\ 
			&=
			\frac{fg \partial_i (h) - nh \partial_i(fg)}{f^{n+1} g^{n+1}}\bfga
			+
			\frac{sh(g\partial_i(f)-f\partial_i(g))-\alpha h f \partial_i(g)}{f^{n+1}g^{n+1}}  \bfga.
		\end{align*}
		By considering the degrees of the polynomials appearing in the numerators, we have that  $\partial_i \Gamma^n\subseteq \Gamma_{n+1}$. Then, $\Gamma^\bullet$ is a filtration compatible  with the Bernstein filtration.
		We observe  that $\dim_{\CC(s)} \Gamma^n\leq \dim_\CC [R]_{\leq (\theta+1)n}$, where the latter is a polynomial on $n$ of degree $d$. Then, $M$ is a holonomic $D(s)$-module.
	\end{proof}
	
	The following property will be useful later on.
\begin{remark} \label{Rem:iso2}
	Let $R=\KK[x_1,\ldots,x_n]$ be a polynomial ring.
	Let $f,g\in R$ be nonzero elements and $h(s)\in R[s]$. We have the following isomorphism of $D_{R|\KK}[s]$-modules $\psi_{\alpha,\ell}:\cM^\alpha_{f/g}[s]\to \cM^\alpha_{f/g}[s]$ defined by 
	\[\frac{h(s)}{f^a g^b}\bfga \mapsto  
	\frac{h(s-\ell ) }{f^{a+ \ell } g^{b-\ell }}\bfga =  \frac{h(s-\ell ) }{f^{a } g^{b}}\frac{g^\ell}{f^\ell}\bfga.
	\]
\end{remark}
\subsection{Bernstein-Sato polynomials for meromorphic functions}

We use the theory of Bernstein-Sato ideals introduced by Sabbah \cite{Sabbah} to obtain functional equations in $\cM^\alpha_{f/g}[s]$. In the form that we use it in this work we have:

\begin{theorem}[{\cite{Sabbah}}]\label{Thm:Sabbah}
Let $R$ be either a polynomial ring $R=\KK[x_1,\dots,x_n]$ or a ring of holomorphic functions $R=\CC\{x_1,\dots,x_n\}$. Let $f,g\in R$ be nonzero elements. Then, there exists $b(s_1,s_2)\in \CC[s_1,s_2]\setminus \{0\}$ and $\delta(s_1,s_2)\in D_{R|\KK}[s_1,s_2]$ such that 
\begin{equation}\label{EqSabbah}
\delta(s_1,s_2)  f g \; {\bf f^{s_1}g^{s_2}}=b(s_1,s_2) {\bf f^{s_1}g^{s_2}}.
\end{equation}
\end{theorem}

\begin{remark}
The description of the polynomials $b(s_1,s_2)$  appearing in Equation \ref{EqSabbah} is not well understood despite some recent progress \cite{BVWZ,Maisonobe}.
\end{remark}

\begin{theorem}\label{Thm:GeneralizedFuncEq} 
Let $R$ be either a polynomial ring $R=\KK[x_1,\dots,x_n]$ or a ring of holomorphic functions $R=\CC\{x_1,\dots,x_n\}$. Let $f,g\in R$ be nonzero elements. Then, 
\begin{enumerate}
\item For almost all $\alpha\in\RR_{\geq 0}$,  
there exists $b(s)\in \KK[s]\setminus \{0\}$ and $\delta(s)\in D_{R|\KK}[s]$ such that
$$
\delta(s)  f \bfga= b(s)\bfga.
$$
\item There exists $b(s)\in \KK[s]\setminus \{0\}$ and $\delta(s)\in D_{R|\KK}[s]$ such that
$$
\delta(s)  \frac{f}{g} \bfg= b(s)\bfg.
$$
\end{enumerate}

\end{theorem}

\begin{proof}
Sabbah \cite{Sabbah} and Gyoja \cite{Gyoja} proved that $b(s_1,s_2)$, in Equation \ref{EqSabbah}, can be chosen with the form
\begin{equation*}
	b(s_1,s_2) =\prod_i (a_i s_1 + b_i s_2 + \gamma_i),
\end{equation*} 
where $a_i,b_i\in\mathbb{N}$, $\gcd(a_i,b_i)=1$ and $\gamma_i\in \QQ_{>0}$. 
After evaluating Equation \ref{EqSabbah}
at $s_1=s$ and $s_2=-s-\alpha$, we obtain the following functional equation in $\cM^\alpha_{f/g}[s]$,
\begin{equation*}
(\delta(s,-s-\alpha)g) f \; \bfga=b(s,-s-\alpha) \; \bfga.
\end{equation*}
We note that 
$$
b(s,-s-\alpha)=\prod_i ((a_i  - b_i)s + \gamma_i-\alpha),
$$
which is zero if and only if $a_i=b_i=1$ and $\gamma_i=\alpha$ for some $i$. If $\alpha\neq \gamma_i$ for every $i$, then $b(s,-s-\alpha)\neq 0$. 

On the other hand, after evaluating Equation \ref{EqSabbah}
at $s_1=s$ and $s_2=-s$, we obtain the following functional equation in $\cM_{f/g}[s]$,
\begin{equation*}
	(\delta(s,-s)  g^2 )\; \left( \frac{f}{g} \right) \; {\bf \frac{f^{s}}{g^{s}}}=b(s,-s) \;  {\bf \frac{f^{s}}{g^{s}}}.
\end{equation*}
In particular $b(s,-s)\neq 0$ and all its roots are rational numbers. 
\end{proof}

%
%

\begin{definition}\label{DefBSMero}
	Let $R$ be either a polynomial ring $R=\KK[x_1,\dots,x_n]$ or a ring of holomorphic functions $R=\CC\{x_1,\dots,x_n\}$. Let $f,g\in R$ be nonzero elements. 
The Bernstein-Sato polynomial  $b_{f/g}^\alpha(s) \in \KK[s]$ of order $\alpha$ of the meromorphic function $f/g$, if it exists, is the monic polynomial of smallest degree satisfying the first functional equation in Theorem \ref{Thm:GeneralizedFuncEq}. Analogously, 
the Bernstein-Sato polynomial  $b_{f/g} (s) \in \KK[s]$ of the meromorphic function $f/g$ is the monic polynomial of smallest degree satisfying the second functional equation in Theorem \ref{Thm:GeneralizedFuncEq}.
%
%
%
\end{definition}

%
%

We note that, because of the way we build $b(s)$ in Theorem \ref{Thm:GeneralizedFuncEq}, it is not clear that it has rational roots if $\alpha\not\in\mathbb{Q}$. We will address this issue in Section \ref{Rational}.

Regarding the polynomial $b(s_1,s_2)$ from \eqref{EqSabbah}, Cassou-Nogu\`es and Libgober \cite{CNL} (see also \cite{Alv22}) proved that it is divisible by a product of linear forms defining the jumping walls of the mixed multiplier ideals $\mathcal{J}(f^{\lambda_1}g^{\lambda_2})$ (see \cite{LazBook2}, \cite{AAD} for unexplained terminology). Work in  \cite[Example 5.8]{AADG2} show that there are  explicit examples where $b(s,-s-\alpha)=0$ with $\alpha={\rm lct}_0(f)={\rm lct}_0(g)$  being the log-canonical thresholds of $f$ and $g$. However we should point out that the conditions needed for this to happen are very restrictive.


As in the classical case, we have the following characterization:

%

\begin{proposition}
Let $R$ be either a polynomial ring $R=\KK[x_1,\dots,x_n]$ or a ring of holomorphic functions $R=\CC\{x_1,\dots,x_n\}$. Let $f,g\in R$ be nonzero elements and let $\alpha\in\RR_{\geq 0}$ such that $b_{f/g}^\alpha(s)$ exist. Then the following are equal:
\begin{enumerate}
		\item The Bernstein-Sato polynomial of $f/g$ of order $\alpha$;
		\item The minimal polynomial of the action of $s$ on $\displaystyle \frac{D_{R|\KK}[s] \bfga}{D_{R|\KK}[s] f \bfga}$;
		\item The monic element of smallest degree in $\KK[s] \cap (\mathrm{Ann}_{D_{R|\KK}[s]}(\bfga) + D_{R|\KK}[s] f)$.
	\end{enumerate}

\end{proposition}

\begin{proof}
	The equality between $(1)$ and $(2)$ follows from the definition.  The first functional equation of Theorem \ref{Thm:GeneralizedFuncEq} gives us
	$$\delta(s)f  - b(s) \in \mathrm{Ann}_{D_{R|\KK}[s]}\left(\bfga \right)$$ and thus $b(s) \in \mathrm{Ann}_{D_{R|\KK}[s]}(\bfga)+ D_{R|\KK}[s] f$ and $(3)$ is also equal. For the equality between $(2)$ and $(3)$, we observe that
	
	\begin{align*}
	 \frac{D_{R|\KK}[s] \bfga}{D_{R|\KK}[s] {f}  \bfga}  & \cong \mathrm{coker}\left( \frac{D_{R|\KK}[s]}{\mathrm{Ann}_{D_{R|\KK}[s]}({ \bfga})} \xrightarrow{\cdot {f} } \frac{D_{R|\KK}[s]}{\mathrm{Ann}_{D_{R|\KK}[s]}(\bfga)} \right) \\
	 & \cong \frac{D_{R|\KK}[s]}{\mathrm{Ann}_{D_{R|\KK}[s]}(\bfga) + D_{R|\KK}[s] {f} }. 
	\end{align*} 
	
%
	
\end{proof}

We can get an analogous characterisation of $b_{f/g}(s)$ by using similar arguments.

\begin{example}[Separated variables]
	Take two sets of independent variables $x=(x_1,\ldots,x_k)$ and $y=(y_1,\ldots,y_l)$, i.e. $x_i\neq y_j$ for every $i=1,\ldots,k$ and $j=1,\ldots,l$. 
	Let $b_f(s)$ be the Bernstein- Sato polynomial of $f(x)$, and let $P(s,x,\partial_{x})$ be a differential operator such that
	\[P(s,x,\partial_{x})f^{s+1}=b_f(s)f^s.\]
	Then for any $\alpha\in\RR_{\geq 0}$ and $g(y)$ we have
	\[P(s,x,\partial_{x})f \bfga=b_f(s) \bfga,
	\]
	and thus $b_f(s)\mid b_{f/g}^{\alpha}(s)$.
\end{example}
\begin{example}\label{Exam:monomials}
	As a particular case of the preceding example, consider the function
	\[
	h(x_1,\ldots,x_n)=\frac{x_1^{m_1}\cdots x_k^{m_k}}{x_{k+1}^{m_{k+1}}\cdots x_n^{m_n}}.\] 
	Then for every $\alpha\in\RR_{\geq 0}$ the polynomial
	$b^\alpha_h(s)$ equals
	\[\prod_{i=1}^{k}\left(\prod_{j=1}^{m_i}\left(s+\frac{j}{m_i}\right)\right). \]
\end{example}

\begin{remark}\label{RmkNero}
	Since the operator $\delta(s)=1$ verifies the functional equation \[\delta(s)  \frac{1}{g^{s+\alpha}} = 1 \frac{1}{g^{s+\alpha}},\] we have that $b^\alpha_{1/g}(s)=1$  for  every $\alpha\in\RR_{\geq 0}$, and thus the Bernstein-Sato polynomial of $1/g$ has no roots. However, there exist non constant functions $g$ for which the topological zeta function of $1/g$ has non trivial poles \cite{GVLLem}. This phenomenon shows that a direct generalization of the strong monodromy conjecture in the case of meromorphic functions is not possible and deserves further investigation.
\end{remark}

We now show that for any value of $\alpha$, $-1$ is always a root of $b^\alpha_{f/g}(s)$ for most cases.
\begin{proposition}\label{PropNegOneIsRoot}
Let $R$ be either a polynomial ring $R=\KK[x_1,\dots,x_n]$ or a ring of holomorphic functions $R=\CC\{x_1,\dots,x_n\}$. Let $f,g\in R$ be nonzero elements such that $f$ is a nonzero divisor in $R/gR$ nor a unit. Then,  $b^\alpha_{f/g}(-1)=0$ for every $\alpha\in \RR_{\geq 0}$ such that $b^\alpha_{f/g}(s)$ exists.
\end{proposition}
\begin{proof}
We fix $\alpha\in \RR_{\geq 0}$. We set $M^\alpha_{f/g}= R_{fg} \ga$. This free module has an structure of $D_{R|\KK}$-module given by
\[
\partial \cdot \frac{h}{f^a g^b}\ga =   
\partial \left( \frac{h}{f^a g^b} \right)\ga 
-\alpha \frac{\delta(g)}{f^ag^{b+1}}\ga,
\]
where $\partial$ is a $\KK$-linear derivative on $R$.
There is a specialization  map of $D_{R|\KK}$-modules from $\cM^\alpha_{f/g}[s]\to M^\alpha_{f/g}$ by sending  $s\mapsto -1$. From the first part of Theorem \ref{Thm:GeneralizedFuncEq}, there exists $\delta(s)\in D_{R|\KK}[s]$ such that
\begin{equation*}
\delta(s)  f \bfga= b(s)\bfga.
\end{equation*}
After specializing, 
we have 
\begin{equation}\label{Eq -1 2}
\delta(-1)   \ga= b(-1)\frac{1}{f}\ga.
\end{equation}
Note that $\delta(-1)   \ga\in R_{g} \ga$. Then, there exists $u\in R$ and $\theta\in \NN$ such that
$\delta(-1)   \ga=\frac{u}{g^\theta}\ga$. Then Equation \ref{Eq -1 2} implies
$\frac{u}{g^\theta}=b(-1)\frac{1}{f}$, which is equivalent to $fu=b(-1) g^\theta$. 
If $\theta=0$, we have that $b(-1)=0$, because $f$ is not a unit.
Since $f$ is a nonzero divisor in $R/gR$, it is also a nonzero divisor in  $R/g^{\theta}R$ if $\theta \neq 0.$
Thus  $b^\alpha_{f/g}(-1)=0$.
\end{proof}

\begin{corollary}
Let $R$ be either a polynomial ring $R=\KK[x_1,\dots,x_n]$ or a ring of holomorphic functions $R=\CC\{x_1,\dots,x_n\}$. Let $f,g\in R$ be nonzero elements such that $f$ does not divide $g$. Then,  $b^\alpha_{f/g}(-1)=0$ for every $\alpha\in \RR_{>0}$ such that $b^\alpha_{f/g}(s)$ exists.
\end{corollary}
\begin{proof}
After removing common factors we can assume that $f$ and $g$ are relatively prime. 
If $g$ is a unit, this is a known fact of the classical Bernstein-Sato polynomial.
Then, 
the claim  follows  from Proposition \ref{PropNegOneIsRoot}, because $f$ and $g$ form a regular sequence. We note that $f$ is not a unit because it  does not divide $g$.
\end{proof}

\subsection{Rationality of the roots of the meromorphic Bernstein-Sato polynomial.} \label{Rational}

A classical result of Kashiwara \cite{Kashiwara}   asserts that the roots of the Bernstein-Sato polynomial are negative rational numbers. This result was later refined by Lichtin \cite{Lic89} by giving a set of candidate roots described in terms of the numerical data of the log resolution of the singularity. A similar result also holds in the meromorphic case for those $\alpha\in \RR_{\geq 0}$ such that $b_{f/g}^\alpha(s)$ exist. 

Let $ \cO_{\CC^n,0}$ be the ring of germs of holomorphic functions around a point $0\in \CC^n$, which we identify with $R= \CC\{ x_1,\dots , x_n\}$ by taking local coordinates. We take a small neighborhood of the origin  $U \subseteq \CC^n$ where $f$ and $g$ are holomorphic. Let 
$J_f=(\frac{df}{dx_1}, \dots , \frac{df}{dx_1})$ and $J_g=(\frac{dg}{dx_1}, \dots , \frac{dg}{dx_1})$ be the Jacobian ideals of $f$ and $g$ respectively. We may assume that the zero locus of $J_f$ is contained in the zero locus of $f$ and the same condition holds for $g$.

 Let $\pi : X \to U$ be a  log resolution of the meromorphic germ   $f / g$, which is in particular a log resolution of $ f^{-1}(0) \cup g^{-1}(0)$.
 In the sequel we use the notations in Section \ref{Sec:log_resolution}, but just recall that we associated to $f / g$ the divisor $\tilde{F} = \tilde{F}_0 + \tilde{F}_\infty$ where  $$\tilde{F}_0=\sum_{i \in I_0} N_{f/g,i}E_i \hskip 5mm \textrm{and} \hskip 5mm  \tilde{F}_\infty=\sum_{i \in I_\infty} N_{f/g,i} E_i. $$ 

{  In this subsection, we give a $D_{R[t]|\KK}$-module structure on $\cM^\alpha_{f/g}[s]$ where the new variable $t$ acts as multiplication by $f$. We define
$$
t\cdot \frac{h(s)}{f^a g^b}\bfga =   
\frac{h(s+1)}{f^a g^b} f \bfga,
$$
and 
$$
\partial_t \cdot \frac{h(s)}{f^a g^b}\bfga =   
\frac{h(s-1)}{f^a g^b} \frac{1}{f} \bfga.\\
$$

\noindent A simple computation shows that $\partial_t  t - t \partial_t =1$ and that $-\partial_t t$ acts as multiplication by $s$. Moreover $ts -st=t$ and we may consider the ring $D_{R|\KK}\langle s,t\rangle\subseteq D_{R[t]|\KK}$. Denote for simplicity
 $$ \cN_{f/g}^\alpha = D_{R|\KK}[s] \bfga \subseteq \cM^\alpha_{f/g}[s]. $$
We have that $\cN^\alpha_{f/g}$ is indeed a 
$D_{R|\KK}\langle s,t\rangle$-module.  As in the classical case, we may view the Bernstein-Sato polynomial of $f/g$ of order $\alpha$ as the minimal polynomial of the action of $-\partial_t t =s$ on 
$$\dfrac{\cN^\alpha_{f/g}}{ t \; \cN^\alpha_{f/g}}.$$}

\begin{theorem}\label{Thm:Kashiwara}

Let $R= \CC\{ x_1,\dots , x_n\}$ be the ring of holomorphic functions and consider $\alpha\in\RR_{\geq 0}$ such that $b_{f/g}^\alpha(s)\neq 0$ for  nonzero elements $f,g\in R$. 
The roots of $b_{f/g}^\alpha(s)$ are negative rational numbers. Moreover, they are contained in the set
$$\Bigg\{ \frac{k_i+1+ \ell }{N_{f/g, i}}  \;  ; \;  \ell \geq 0, \;  i\in I_0\Bigg\}.$$
\end{theorem}
\begin{proof}
Let $\pi : X \to U$ be a  log resolution of the meromorphic germ  $f / g$ and let $D_X$ and $D$ denote the corresponding rings of differential operators.  Assume that in local coordinates around a point $ p\in X$ we have either
$$\frac{\tilde{f}}{\tilde{g}} =  \;   z_1^{N_{1}} \cdots z_m^{N_{m}} \hskip 5mm \textrm{or} \hskip 5mm  \frac{\tilde{f}}{\tilde{g}} =   \;  \frac{1}{z_1^{N_1} \cdots z_m^{N_m}}.$$
Therefore, from Example \ref{Exam:monomials} and Remark \ref{RmkNero}, locally at the point $p$ the Bernstein-Sato polynomial of order $\alpha$ is either 
$$b_{\tilde{f}/\tilde{g}}^\alpha(s)= \prod_{i=1}^{m}\left(\prod_{j=1}^{N_i}\left(s+\frac{j}{N_i}\right)\right) \hskip 5mm \textrm{or} \hskip 5mm b_{\tilde{f}/\tilde{g}}^\alpha(s)=1.$$
Consider a cover of $X$ by affine open subsets. On each of these subsets we can use local coordinates as before and get the corresponding local Bernstein-Sato polynomial with negative rational roots. The global Bernstein-Sato polynomial, that we also denote by $b_{\tilde{f}/\tilde{g}}^\alpha(s)$ if no confusion arises, is the least common multiple of such local Bernstein-Sato polynomials.

Then the rationality result boils down to proving the existence of an integer $\ell\geq 0$ such that 
\begin{equation}\label{Eq: multiplicative}
 b_{f/g}^\alpha(s)  \; \vert \; b_{\tilde{f}/\tilde{g}}^\alpha(s) b_{\tilde{f}/\tilde{g}}^\alpha(s+1) \cdots b_{\tilde{f}/\tilde{g}}^\alpha(s + \ell).\end{equation}

The proof of this fact in the meromorphic case is analogous to the classical case considered by Kashiwara. We only have to be careful when 
dealing with the module
$$ \cN_{\tilde{f}/\tilde{g}}^\alpha = D_{X}[s] \bfgt.$$

If $\frac{\tilde{f}}{\tilde{g}} =  \;   z_1^{N_{1}} \cdots z_m^{N_{m}}$  we have that $\cN_{\tilde{f}/\tilde{g}}^\alpha $ is subholonomic as $D_X$-module as in the classical case. Specifically, its characteristic variety  is the closure  of $\{(x, sdf(x)) \;  ; \;  f(x)\neq 0 , s\in \CC \}$ in the cotangent bundle $T^\ast X$ and thus it has dimension $ n + 1$. 
$\cN_{\tilde{f}/\tilde{g}}^\alpha $ has the same  characteristic variety when $\frac{\tilde{f}}{\tilde{g}} =   \;  \frac{1}{z_1^{N_1} \cdots z_m^{N_m}}$ so it is subholonomic as well.

The rest of the proof follows the same lines of reasoning of Kashiwara so we  just sketch the key points and refer to his work  for details \cite{Kashiwara}. Let 
\[\cN= \mathcal{H}^0 (\pi_+ \cN_{\tilde{f}/\tilde{g}}^\alpha)= \mathcal{H}^0 (\pi_\ast( D_{U\leftarrow X} \otimes_D^{\bf L} \cN_{\tilde{f}/\tilde{g}}^\alpha))\] 
be the degree zero direct image of  $\cN_{\tilde{f}/\tilde{g}}^\alpha $, where $ D_{U\leftarrow X}$ denotes the transfer bimodule. There is a canonical section $u \in \cN$ associated  to 
$ {\bf 1}_{U\leftarrow X} \otimes \bfgt \in  D_{U\leftarrow X} \otimes_D^{\bf L} \cN_{\tilde{f}/\tilde{g}}^\alpha$ that allows us to describe a $D$-submodule
$$  \cN'= D[s] u \subseteq \cN.$$
This inclusion is indeed a morphism  of $D\langle s,t\rangle$-modules. We have that  $\cN$ is subholonomic as $D$-module \cite[Lemma 5.7]{Kashiwara} and moreover $\cN / t \cN$ is holonomic so the action of $s$ has a minimal polynomial which we denote $b_{\cN}(s)$. We have that $\cN' / t \cN'$ is holonomic as well which leads to a polynomial $b_{\cN'}(s)$. Using that $\cN / \cN'$ is also holonomic,   there exist $\ell \geq 0$ large enough such that $t^\ell \cN \subseteq \cN'$ \cite[Proposition 5.11]{Kashiwara}. From the relation
$$b_{\cN}(s+j) t^j \cN= t^j b_{\cN}(s) \cN \subseteq t^{j+1}\cN,$$  $j\geq 0$, we obtain $$b_{\cN}(s+\ell) \cdots b_{\cN}(s) \cN' \subseteq b_{\cN}(s+\ell) \cdots b_{\cN}(s) \cN \subseteq t^{\ell +1} \cN \subseteq t \cN'$$ and thus $b_{\cN'}(s) \; \vert \; b_{\cN}(s+\ell) \cdots b_{\cN}(s)$. Finally, we have to relate this equation to (\ref{Eq: multiplicative}).
On the one hand, since $ b_{\tilde{f}/\tilde{g}}^\alpha(s) \cN_{\tilde{f}/\tilde{g}}^\alpha  \subseteq t  \cN_{\tilde{f}/\tilde{g}}^\alpha $ there exist a $D_X$-module endomorphism $h: \cN_{\tilde{f}/\tilde{g}}^\alpha  \rightarrow \cN_{\tilde{f}/\tilde{g}}^\alpha $ such that $b_{\tilde{f}/\tilde{g}}^\alpha(s) = t \circ h$. Applying the functor $\mathcal{H}^0 \pi_+$ we get 
$$b_{\tilde{f}/\tilde{g}}^\alpha(s) \cN = t \circ \mathcal{H}^0( \pi_+(h) \cN) \subseteq t \cN$$ 
and thus $b_{\cN}(s) \; \vert \; b_{\tilde{f}/\tilde{g}}^\alpha(s)$. 
On the other hand we have a surjection of $D\langle s,t\rangle$-modules $\cN' \twoheadrightarrow \cN_{f/g}^\alpha $ and thus $b_{\cN'}(s) \cN_{f/g}^\alpha \subseteq t \cN_{f/g}^\alpha$ which gives  $b_{f/g}^\alpha(s)) \; \vert \;  b_{\cN'}(s)$ and Equation (\ref{Eq: multiplicative}) follows.

The refinement given by Lichtin \cite{Lic89} requires to pass from left to right $D$-modules in order to plug in all the information given by the relative canonical divisor. Namely, for a left $D$-module $\cM$ we consider the right $D$-module $\cM^{(r)}:=\omega_{R} \otimes_R \cM$. In local coordinates, this operation is given by the involution on $D$ that sends a differential operator $\delta$ to its adjoint $\delta^\ast$ described uniquely by the properties $(\delta \delta')^\ast = \delta'^\ast \delta^\ast$, $f^\ast = f $ for any $f\in R$ and $\partial_i ^\ast = -\partial_i$. This gives an equivalence between left and right $D$-modules. We can extend this to an equivalence between left and right $D\langle s, t \rangle$-modules by taking $t^\ast =t$, $\partial_t^\ast = -\partial_t$ and $s^\ast= (-\partial_t t)^\ast = t\partial_t = \partial_t t -1 = -s-1$.
For an element $u$ in a $D\langle s, t \rangle$-module $\cM$ we set $u^\ast = dx_1\wedge \cdots \wedge d{x_n} \otimes u$. Then we have $u^\ast \delta^\ast = (\delta u)^\ast$ for any $\delta \in D\langle s, t \rangle$. 

Now, for $ u= \bfga \in  \cN_{f/g}^\alpha$ we have that 
the functional equation 
$$
\delta(s)  f \bfga= b(s)\bfga
\hskip 5mm 
\textrm{becomes}
\hskip 5mm
f \left(\bfga\right )^\ast (\delta(s))^\ast = (b(s))^\ast \left(\bfga \right)^\ast,
$$
and the minimal polynomials satisfying these equations are related by
\begin{equation}\label{Eq:dualBS}
 b_{f/g}^\alpha(s) = b_u(s)= b_{u^\ast}(-s-1)= b_{(f/g)^\ast}^\alpha (-s-1).
 \end{equation} In particular $ b_{u^\ast}(s)$ is the minimal polynomial of the action of $s$ on $(\cN_{f/g}^\alpha)^{(r)} / (\cN_{f/g}^\alpha)^{(r)} t$.

Let $\pi : X \to U$ be a  log resolution of   $f / g$. We may construct  $(\cN_{\tilde{f}/\tilde{g}}^\alpha)^{(r)} = \omega_X \otimes_{\cO_X}  \cN_{\tilde{f}/\tilde{g}}^\alpha$  as before but, following Lichtin, we consider the $D_X\langle s,t \rangle$-submodule
$$\cN_v:=  v D_X[s], $$  with $v= \pi^\ast(dx_1\wedge \cdots \wedge dx_n) \otimes \bfgt \in  \left(\cN_{\tilde{f}/\tilde{g}}^\alpha\right)^{(r)}$. In local coordinates  around a point $ p\in X$ we have either
$$v =  \; w   z_1^{N_{1}s +k_1} \cdots z_m^{N_{m} s +k_m} \hskip 5mm \textrm{or} \hskip 5mm  v =   \; w  z_1^{-N_1s +k_1} \cdots z_m^{-N_m s +k_m},$$ where $w$ is a unit. The Bernstein-Sato polynomial associated to these elements is either 
$$b_{v}(s)= \prod_{i=1}^{m}\left(\prod_{j=1}^{N_i}\left(s+\frac{k_i +j}{N_i}\right)\right) \hskip 5mm \textrm{or} \hskip 5mm b_{v}(s)=1,$$ 
and we have a relation with the Bernstein-Sato polynomial associated to $\cN_v$, i.e. the minimal polynomial of the action of $s$ on $\cN_v / \cN_ v t$, given by $b_{\cN_v}(s)= b_v(-s-1)$.  By considering a cover of $X$ we get the global Bernstein-Sato polynomial that we also denote by $b_{\cN_v}(s)$ if no confusion arises.

Then the refined version of Lichtin result follows from proving the existence of an integer $\ell\geq 0$ such that the polynomial $b_{u^\ast}(s)$ in Equation (\ref{Eq:dualBS}) satisfies
\begin{equation}\label{Eq: multiplicative2}
b_{u^\ast}(s)  \; \vert \; b_{\cN_v}(s) b_{\cN_v}(s-1) \cdots b_{\cN_v}(s - \ell).
\end{equation}
The rest of the proof is analogous to Kashiwara's proof of rationality but working in the category of right $D$-modules. 

\end{proof}

\begin{corollary}\label{Cor:Kashiwara}
Let $f,g\in R= \CC\{ x_1,\dots , x_n\}$ be  nonzero holomorphic functions. 
The roots of $b_{f/g}(s)$ are negative rational numbers. Moreover, they are contained in the set
$$\Bigg\{ \frac{k_i+1+ \ell }{N_{f/g, i}}  \;  ; \;  \ell \geq 0, \;  i\in I_0\Bigg\}.$$
\end{corollary}

\begin{proof}
It follows from the fact that $b_{f/g}(s)$ divides $b_{f/g}^1(s)$.
\end{proof}

\subsection{Takeuchi's meromorphic Bernstein-Sato polynomial}

Takeuchi \cite{takeuchi2023}  independently introduced a notion of Bernstein-Sato polynomials for meromorphic functions which differs from the one in Definition \ref{DefBSMero}. When $R$ is a polynomial ring $R=\KK[x_1,\dots,x_n]$ or a ring of holomorphic functions $R=\CC\{x_1,\dots,x_n\}$, he defines the Bernstein-Sato polynomial  $b_{f/g,\alpha}^{\rm mero}(s)$  of order $\alpha$ (there $\alpha$ is a non negative integer) of the meromorphic function $f/g$ as the monic polynomial of smallest degree satisfying the functional equation
\begin{equation} \label{Take1}
\delta_1(s) \frac{f}{g}  {{\bf \frac{f^s}{g^{s+\alpha}}}} + \cdots + \delta_\ell(s) \frac{f^\ell}{g^\ell}  {{\bf \frac{f^s}{g^{s+\alpha}}}}= b(s)   {{\bf \frac{f^s}{g^{s+\alpha}}}},
\end{equation}
for some $\delta_i(s)\in D_{R|\KK}[s]$ and $\ell\geq 1$ large enough.

As Takeuchi  points out \cite[Discussion in \textsection 1]{takeuchi2023},  one has that 
\begin{equation}\label{FunctionalEqTakeuchi}
b_{f/g,\alpha}^{\rm mero}(s) \hskip 2mm | \hskip 2mm b_{f/g}^{\alpha}(s).
\end{equation}

We now give a different proof of the existence of the Bernstein-Sato polynomial for a meromorphic function as in Equation \ref{Take1}. Our proof relies in the holonomicity  of the module  $\cM^\alpha_{f/g}(s)$, cf. Lemma \ref{Lemma:Holonomicity}, and not on $V$-filtrations as the original work \cite{takeuchi2023}.

\begin{theorem}[{\cite[Theorem 1.1]{takeuchi2023}}]\label{Thm:Takeuchi}
	Let $R=\KK[x_1,\ldots,x_n]$.
	Let $f,g\in R$ be nonzero elements and take $\alpha\in\RR_{\geq 0}$.  
	There exists $b(s)\in \KK[s]\setminus \{0\}$,  $\delta_i(s)\in D_{R|\KK}[s]$ and $\ell\geq 1$ large enough such that
	\begin{equation*}
		\delta_1(s) \frac{f}{g}  \bfga + \cdots + \delta_\ell(s) \frac{f^\ell}{g^\ell} \bfga= b(s)   \bfga.
	\end{equation*}
\end{theorem}

\begin{proof}
	Recall from Definition \ref{Def:Module_M_alpha} that $\cM^\alpha_{f/g}[s]= R_{fg}[s] \bfga$, $D(s)=D_{R(s)| \KK(s)}$, and $\cM^\alpha_{f/g}(s)=\cM^\alpha_{f/g}[s]\otimes_{\KK[s]}\KK(s)$. We consider the following collection of $D(s)$-modules
	\[N_t:= D(s) \cdot  \biggl\{ \frac{f^j}{ g^j } \bfga \; | \; j\geq t\biggr\}.
	\]
	By Lemma \ref{Lemma:Holonomicity}, the chain $N_1 \supseteq N_2 \supseteq N_3 \supseteq\ldots$
	stabilizes and thus, there exists $t$ such that $N_t=N_{t+1}$. Therefore
	$$\frac{f^t}{g^t} \bfga \in D(s) \cdot  \biggl\{ \frac{f^j}{g^j} \bfga \; | \; j > t\biggr\},$$ and there exist $\tilde{\gamma}_1(s), \dots, \tilde{\gamma}_\ell(s) \in D(s)$, with $\ell\geq 1$ large enough, such that 
	\begin{equation*} \label{Take2}
		\tilde{\gamma}_1(s) \frac{f^{t+1}}{g^{t+1}} \bfga   + \cdots + \tilde{\gamma}_\ell(s) \frac{f^{t+\ell}}{g^{t+\ell}} \bfga =   \frac{f^t}{g^t} \bfga. 
	\end{equation*}
	
	Consider $\widetilde{ \delta}_i(s)=\widetilde{ b}(s)\widetilde{\gamma}_i(s)$, where $\widetilde{b}(s) \in \KK[s]\setminus\{0\}$ is the minimum common multiple of the denominators of  $\widetilde{\gamma}_i(s)$ and thus  $\widetilde{  \delta}_i(s)\in  D_{R|\KK}[s]$. Then we have 
	\begin{equation*} \label{Take3}
		\tilde{\delta}_1(s) \frac{f^{t+1}}{g^{t+1}} \bfga   + \cdots + \tilde{\delta}_\ell(s) \frac{f^{t+\ell}}{g^{t+\ell}} \bfga =  \widetilde{ b}(s) \frac{f^t}{g^t} \bfga. 
	\end{equation*}
	Applying the isomorphism $\psi_{\alpha,\ell}$ defined in Remark \ref{Rem:iso2}, we obtain
	\begin{equation*} 
		\tilde{\delta}_1(s-t) \frac{f}{g}\bfga   + \cdots + \tilde{\delta}_\ell(s-t) \frac{f^\ell}{g^\ell} \bfga =  \widetilde{b}(s-t)  \bfga. 
	\end{equation*}
	Setting $\delta_i(s)= \widetilde{\delta}_i(s-t) $ and $b(s)= \widetilde{b}(s-t)$ we get the desired result.
	
\end{proof}

\section{Archimedean Local Zeta Functions}\label{Sec4}

In this section we use the Bernstein-Sato polynomial of $f/g$ to study Archimedean local zeta functions.  Let $ \cO_{\CC^n,0}$ be the ring of germs of holomorphic functions around a point $0\in \CC^n$, which we identify with $R= \CC\{ x_1,\dots , x_n\}$ by taking local coordinates. For definiteness we take a small neighborhood of the origin  $U \subseteq \CC^n$ where $f$ and $g$ are both holomorphic. By a test function we mean a smooth function with compact support on $\mathbb{C}^n$. Then the local zeta function attached to a test function $\phi$ and $f/g$ is defined as
\[
Z_{\phi}(s,f/g)=\bigintssss_{U\setminus H}\phi(x) \left\vert \frac{f(x)}{g(x)}\right\vert ^{2s}\ \mathrm{d}x,
\]
where $s$ is a complex number and $H=f^{-1}(0)\cup g^{-1}(0)$. When $g=1$ (the classical case), local zeta functions were introduced by Gel'fand and Shilov in the 50's  \cite{Gel-Shi}. In that case is not difficult to show that $Z_{\phi}(s,f/1)=Z_{\phi}(s,f)$ converges on the half plane $\{s\in\mathbb{C}\ ; \operatorname{Re}(s)>0 \}$ and defines a holomorphic function there. In contrast, the convergence of the parametric integral $Z_{\phi}(s,f/g)$ is a delicate matter as Veys and Z\'u\~niga-Galindo noted  \cite[Remark 2.1 (1)]{Ve-Zu}. In particular they explain that the convergence of the integral does not follow from the fact that $\phi$ has compact support and they use an embedded resolution of  $H$ to show that the integral has a region of convergence \cite[Theorem 3.5 (1)]{Ve-Zu}. As in the classical case, the meromorphic continuation of $Z_{\phi}(s,f/g)$ does not depend on the set $U$.

Building on the properties of the log canonical threshold of $f$ and $g$ we describe  a simple region of convergence for $Z_{\phi}(s,f/g)$. First, we recall 
the definition of the log canonical threshold of $f$ and some of its properties regarding  the local zeta function of $f$ and $\phi$.
\begin{definition}\label{Def:logCan}
	Let $f$ be a holomorphic function in an open set $U\subseteq \mathbb{C}^n$. The log-canonical threshold of $f$ (at the origin) is the number
	\[\lct_0(f)=\sup\left\{\lambda \in \RR_{>0}\ ;\ \int_{B_\varepsilon(0)}  |f|^{-2\lambda} < \infty, \text{ for some } \varepsilon>0 \right\}.\]
\end{definition}

It is known that $\lct_0(f)$ is a positive rational number. Furthermore,  we have that 	$Z_{\phi}(s,f)$ is a holomorphic function on $\operatorname{Re}(s)>-\lct_0(f)$. Using this fact one may show that $Z_{\phi}(s,g^{-1})=\int_{U\setminus g^{-1}(0)}\phi(x) \left\vert g(x)\right\vert ^{-2s}\ \mathrm{d}x$, is holomorphic on $\operatorname{Re}(s)<\lct_0(g)$.

\begin{lemma}\label{Lem:ConvergenceZetaF}
Let $R= \CC\{ x_1,\dots , x_n\}$ be the ring of holomorphic functions and let $f,g\in R$ be  nonzero elements. Then, the integral $Z_{\phi}(s,f/g)$ converges for $-\lct_0(f)< \operatorname{Re}(s)<\lct_0(g)$. Furthermore, it defines a holomorphic function there.
\end{lemma}
\begin{proof}
	First we show that $Z_{\phi}(s,f/g)$ is finite on the interval  $-\lct_0(f)< \operatorname{Re}(s) \leq 0$. Note that in this region the function $|g(x)|^{-2s}$ is well defined and continuous over $U\setminus H$. Then
	\[\left|\bigintssss_{U\setminus H}\phi(x) \left\vert \frac{f(x)}{g(x)}\right\vert ^{2s}\ \mathrm{d}x\right|\leq \bigintssss_{U\setminus H}|\phi(x)| \left\vert \frac{f(x)}{g(x)}\right\vert ^{2\operatorname{Re}(s)}\ \mathrm{d}x<\infty,\]
	since the the support of $\phi$ is compact and $\left\vert \frac{f(x)}{g(x)}\right\vert ^{2\operatorname{Re}(s)}$ is continuous. A symmetric argument shows that $Z_{\phi}(s,f/g)<\infty$ for  $0\leq \operatorname{Re}(s)<\lct_0(g)$. The fact that $Z_{\phi}(s,f/g)$ defines a holomorphic function on  $-\lct_0(f)< \operatorname{Re}(s)<\lct_0(g)$ can be proved as in the classical case \cite[Theorem 3.1]{IgusaOld} by showing that  $\partial/\partial \bar{s}(Z_{\phi}(s,f/g))=0$, which follows in particular from Lebesgue's dominated convergence theorem.
\end{proof}
\begin{remark} 
	The region of convergence of Lemma \ref{Lem:ConvergenceZetaF} is in general smaller that the region of convergence described in the work of Veys and Zu\~niga-Galindo \cite{Ve-Zu}. According to their Example 3.13 (4), the integral $Z_{\phi}(s,y^2+x^4/x^2+y^4)$ converges in  $-3/2< \operatorname{Re}(s)<3/2$, whereas Lemma \ref{Lem:ConvergenceZetaF} shows that the integral converges in $-3/4< \operatorname{Re}(s)<3/4$. We point out that the region of convergence of Veys and Zu\~niga-Galindo is optimal in the sense that the endpoints give actual poles for $Z_{\phi}(s,f/g)$, i.e. $-3/2$ and $3/2$ are poles of $Z_{\phi}(s,y^2+x^4/x^2+y^4)$.
\end{remark}

Another property of $Z_{\phi}(s,f)$ that one may also study is the existence of a meromorphic continuation to the whole complex plane. This fact was conjectured by Gel'fand in the 50's and proved by Bernstein and Gel'fand \cite{Ber-Gel}, Atiyah  \cite{Ati} using resolution of singularities, and  Bernstein  \cite{Ber72} using the Bernstein-Sato polynomial. In the case of meromorphic functions, Veys and  Z\'u\~niga-Galindo \cite[Theorem 3.5 (2)]{Ve-Zu}  showed   that $Z_{\phi}(s,f/g)$ has a meromorphic continuation to the whole complex plane and its poles are described by means of the numerical data of an embedded resolution of singularities of $H$. By using the theory of Bernstein-Sato polynomials developed in Section \ref{Sec3} we present a proof of the meromorphic continuation of $Z_{\phi}(s,f/g)$, resembling the original proof of Bernstein \cite{Ber72}, giving in addition a shorter list of candidate poles.
 
For a differential operator $\delta(s)\in D_{R|\CC}[s]$ we denote the \textit{conjugate operator} by $\overline{\delta}(s)$. This is the operator obtained from $\delta$ by replacing $x_i$ with $\overline{x_i}$ and $\partial_{x_i}$ with $\partial_{\overline{x_i}}$. A simple calculation shows that when the first functional equation in Theorem \ref{Thm:GeneralizedFuncEq} holds one has  
\begin{equation}\label{Eq:IdentityConjugation}
	\delta(s)\overline{\delta}(s) \frac{|f|^{2(s+1)}}{|g|^{2(s+\alpha)}}= b^\alpha_{f/g}(s)^2\frac{|f|^{2s}}{|g|^{2(s+\alpha)}} \text{ in } \cM_{f/g}^\alpha[s].
\end{equation}
\begin{theorem}\label{Thm:MerContZetaF}
{
Let $R= \CC\{ x_1,\dots , x_n\}$ be the ring of holomorphic functions and $f,g\in R$ be  nonzero elements. 
	Let 
	\[\alpha=\max_{k\geq 1}\left\{\alpha^\prime=\frac{\lct_0 g}{k} \ ; \ b^{\alpha^\prime}_{f/g}(s)\neq 0\right\}\quad\text{and}\quad \beta=\max_{k\geq 1}\left\{-\beta^\prime=\frac{\lct_0 f}{k} \ ; \ b^{\beta^\prime}_{g/f}(s)\neq 0\right\}.
	\]
	The local zeta function $Z_{\phi}(s,f/g)$ has a meromorphic continuation to the whole complex plane $\CC$, and  its poles  are contained in the set 
	\[\Bigl\{\zeta-\ell\alpha \ ;\ \ell\in\ZZ_{\geq 0}\Bigr\}_{\zeta \text{ root of } b^\alpha_{f/g}(s)}\bigcup \Bigl\{\ell\beta-\xi \ ;\ \ell\in\ZZ_{\geq 0}\Bigr\}_{\xi \text{ root of } b^\beta_{g/f}(s)}.\]
In particular, the poles of $Z_{\phi}(s,f/g)$ are rational numbers.
}
\end{theorem}
\begin{proof}
	 Let  $b^\alpha_{f/g}(s)\in\CC[s]$ be the nonzero polynomial of the first part of Theorem \ref{Thm:GeneralizedFuncEq}, note that there exists a map of $D_{R|\CC}$-modules $$\psi_\lambda : \cM_{f/g}^\alpha[s]\to \Frac(\CC\{x_1,\ldots,x_d\})$$ 
	that maps $s\mapsto \lambda$ and $\bfga\mapsto \frac{f^\lambda}{g^{\lambda+\alpha}}$ for every $\lambda\in\CC$. We also notice that after possible shrinking $U$,  Equation \eqref{Eq:IdentityConjugation} remains valid for $f,g \in R$, and then multiplication of $Z_{\phi}(s,f/g)$ by $b^\alpha_{f/g}(s)^2$ gives
	\begin{align*}
	b^\alpha_{f/g}(s)^2\,Z_{\phi}(s,f/g) & =\bigintssss_{U\setminus H}\phi(x)  b^\alpha_{f/g}(s)^2\,\left|\frac{f(x)}{g(x)}\right|^{2s}\,\mathrm{d}x\\
	& =\bigintssss_{U\setminus H}\phi(x) |g(x)|^{2\alpha}\,b^\alpha_{f/g}(s)^2 \frac{1}{|g(x)|^{2\alpha}}\left|\frac{f(x)}{g(x)}\right|^{2s}\,\mathrm{d}x\\
	&=\bigintssss_{U\setminus H}\phi(x) |g(x)|^{2\alpha}\,\left(\delta(s)\overline{\delta}(s) \frac{|f(x)|^{2(s+1)}}{|g(x)|^{2(s+\alpha)}}\right)\, \mathrm{d}x\\
	&=\bigintssss_{U\setminus H}\phi(x) |g(x)|^{2\alpha}\, \left( \delta(s)\overline{\delta}(s)  |f(x)|^{2-2\alpha}  \left|\frac{f(x)}{g(x)}\right|^{2(s+\alpha)}\right)\, \mathrm{d}x\\
	&=\bigintssss_{U\setminus H} \left(|f(x)|^{2-2\alpha} \delta^\ast(s)\overline{\delta}^\ast(s)\cdot \phi(x) |g(x)|^{2\alpha}\right)\,  \left|\frac{f(x)}{g(x)}\right|^{2(s+\alpha)}\, \mathrm{d}x\\
	&=Z_{\psi}(s+\alpha,f/g).
	\end{align*}
	{
	Here $\delta^\ast$ denotes the adjoint operator of $\delta$ and  $\psi(x) = |f(x)|^{2-2\alpha} \delta^\ast(s)\overline{\delta}^\ast(s)\cdot \phi(x) |g(x)|^{2\alpha}$ is a complex test function. By Lemma \ref{Lem:ConvergenceZetaF},  $Z_{\psi}(s+\alpha,f/g)$ converges for $-\lct_0 f-\alpha<\operatorname{Re}(s)<\lct_0 g-\alpha$, and thus
	\[
	Z_{\phi}(s,f/g)=\frac{Z_{\psi}(s+\alpha,f/g)}{b^\alpha_{f/g}(s)^2}
	\]
	converges in $-\lct_0 f-\alpha< \operatorname{Re}(s)<\lct_0 g$, outside the possible roots of $b^\alpha_{f/g}(s)$. Multiplying $Z_{\psi}(s+\alpha,f/g)$ by $b^\alpha_{f/g}(s+\alpha)^2$ we get
	\begin{equation*}
		\bigintssss_{U\setminus H} \left(|f(x)|^{2-2\alpha} \delta^\ast(s+\alpha)\overline{\delta}^\ast(s+\alpha)\cdot \psi(x) |g(x)|^{2\alpha}\right)\,  \left|\frac{f(x)}{g(x)}\right|^{2(s+2\alpha)}\, \mathrm{d}x=Z_{\rho}(s+2\alpha,f/g).
	\end{equation*}
	This expression provides a meromorphic extension of $Z_{\phi}(s,f/g)$ to  $-\lct_0 f-2\alpha< \operatorname{Re}(s)<\lct_0 g$ with possible poles in the roots of $b^\alpha_{f/g}(s)\cdot b^\alpha_{f/g}(s+\alpha)$. Iterating this process, we obtain a meromorphic extension to $\operatorname{Re}(s)<\lct_0 g$, for $Z_{\phi}(s,f/g)$.
	} 
	The possible poles of the zeta function are contained in the set
	\[\Bigl\{\zeta-\ell\alpha \ ;\ \ell\in\ZZ_{\geq 0}\Bigr\}_{\zeta \text{ root of } b^\alpha_{f/g}(s)}.\]
	{
	For the continuation `to the right' of  $Z_{\phi}(s,f/g)$ we use the first part of Theorem \ref{Thm:GeneralizedFuncEq} in the following version 
	\[\delta(-s) \frac{g^{-s+1}}{f^{-s+\beta}}=b^\beta_{g/f}(-s) \frac{g^{-s}}{f^{-s+\beta}},\]
	which implies the analogue of Equation \ref{Eq:IdentityConjugation}: $\delta(-s)\overline{\delta}(-s) \frac{|f|^{2(\beta-s)}}{|g|^{2(1-s)}}= b^\beta_{g/f}(-s)^2\frac{|f|^{2(\beta-s)}}{|g|^{2s}}$. If we multiply $Z_{\phi}(s,f/g)$ by $b^\beta_{g/f}(-s)^2$ we obtain
	\[	Z_{\phi}(s,f/g)=\frac{Z_{\phi}(\beta-s,f/g)}{b^\beta_{g/f}(-s)^2},\]
	showing that $Z_{\phi}(s,f/g)$ converges in $-\lct_0 f< \operatorname{Re}(s)<\lct_0 g-\beta$ away from the possible roots of $b^\beta_{g/f}(s)$. Further iteration of this procedure provides the desired conclusion.
	The rationality of the poles of $Z_{\phi}(s,f/g)$ follows from Theorem \ref{Thm:Kashiwara}.
	}
\end{proof}
	
\begin{remark}
	Veys and  Z\'u\~niga-Galindo study local zeta functions for meromorphic functions over local fields of characteristic zero 
	\cite{Ve-Zu}. The statement and proof of Lemma \ref{Lem:ConvergenceZetaF} are also valid in this generality, providing also a simple region of convergence for the non Archimedean  $Z_{\phi}(s,f/g)$, cf.  \cite[Theorem 3.2 (1)]{Ve-Zu}. We do not known if the theory developed in Section \ref{Sec3} is also valid in this generality, but it is certainly true over Archimedean local fields of characteristic zero, that is,  $\mathbb{R}$ or $\mathbb{C}$.  The proof of  Theorem \ref{Thm:MerContZetaF} for the real case can be given by adapting the ideas of our proof and following the lines of Igusa's work \cite[Theorem 5.3.1]{Igusa}.
\end{remark}	
	
\section{Multiplier ideals: analytic construction} \label{Sec5}

In this section we  describe the theory of analytic multiplier ideals for meromorphic functions and relate their jumping numbers to roots of the meromorphic Bernstein-Sato polynomial. 

\begin{definition}
Let $R= \CC\{ x_1,\dots , x_n\}$ be the ring of holomorphic functions and $f,g\in R$ be  nonzero elements such that $f/g$ is not constant.
We define the multiplier ideal (at the origin) of $f/g$ at $\lambda\in \RR_{\geq 0}$ by
$$
\cJ \left( \left( \frac{f}{g}\right)^\lambda\right)=\left\{ h\in R \ ;\ \int_{B_r(0)} \frac{|h|^{2} |g|^{2\lambda}}{|f|^{2\lambda}} < \infty \hbox{ for some }r>0 \right\}, 
$$ where  ${B_r(0)}$ denotes a closed ball of radius $r>0$ around the origin $0$.
\end{definition}

The following useful remark shows that we can describe the meromorphic multiplier ideal from the theory of mixed multiplier ideals $\cJ(f^{\lambda_1} g^{\lambda_2})$ associated to the pair of functions. 

\begin{remark}\label{RemMixedMultMeromorphic}
Consider $t\in\NN$ such that $t\geq \lambda$.
Then,
\begin{align*}
\cJ\left(f^{\lambda} g^{t-\lambda}\right): g^{t}& =\left\{ h\in R \ ;\    hg^{t}\in \cJ\left(f^{\lambda} g^{t-\lambda}\right) \right\} \\
&=\left\{ h\in R \ ;\ \int_{B_r(0)} \frac{|h|^2 |g|^{2t}}{|f|^{2\lambda} |g|^{2t-2\lambda}} < \infty \hbox{ for some } r >0  \right\} \\
&=\left\{ h\in R \ ;\ \int_{B_r(0)} \frac{|h|^2 |g|^{2 \lambda}}{|f|^{2\lambda}} < \infty \hbox{ for some }r>0 \right\} \\
&=\cJ \left(\left( \frac{f}{g}\right)^\lambda\right).
\end{align*}
\end{remark}

\begin{remark}
For every $\lambda\in \RR_{\geq 0}$, we have
$$
\cJ \left( \left( \frac{f}{g}\right)^\lambda\right)= \cJ \left( \left( \frac{g}{f}\right)^{-\lambda}\right).
$$
\end{remark}

\begin{definition}\label{Def:jumping_numbers}
Let $R= \CC\{ x_1,\dots , x_n\}$ be the ring of holomorphic functions and $f,g\in R$ be  nonzero elements 
We say that $\lambda\in \RR_{\geq 0}$ is a jumping number for $f/g$ if for every $\varepsilon>0$ we have either
$$
\cJ \left( \left( \frac{f}{g}\right)^\lambda\right)\neq \cJ \left( \left( \frac{f}{g}\right)^{\lambda-\varepsilon}\right)
\quad \text{or}\quad
\cJ \left( \left( \frac{f}{g}\right)^\lambda\right)\neq \cJ \left( \left( \frac{f}{g}\right)^{\lambda+\varepsilon}\right).
$$
\end{definition}
We stress that multiplier ideals of meromorphic functions are not necessarily continuous on either side. 
For instance, if $g=1$, it is known that
$$
\cJ \left( f^\lambda\right)= \cJ \left( f^{\lambda+\varepsilon}\right)
$$
for small enough $\varepsilon >0$. In contrast, if $f=1$,
$$
\cJ \left( \frac{1}{g^\lambda}\right)\neq \cJ \left( \frac{1}{g^{\lambda-\varepsilon}}\right)
$$
for small enough $\varepsilon >0$.

\begin{proposition}\label{PropSkoda}
Let $R= \CC\{ x_1,\dots , x_n\}$ be the ring of holomorphic functions, $f,g\in R$ be  nonzero elements 
 and consider $\lambda\in \RR_{\geq 0}$.
Then,
$$
\cJ \left( \left( \frac{f}{g}\right)^{\lambda+\ell}\right)
=
\left( \frac{f^\ell}{g^\ell}\left( \cJ  \left( \frac{f}{g}\right)^\lambda\right)\right)\bigcap R
$$
for every $\ell\in\NN.$
In particular,
if  $\lambda+1$ is a jumping number, then $\lambda$ is a jumping number.
\end{proposition}
\begin{proof}
We fix $t\geq \lambda+1$.
The result follows from the equality
$$
\cJ\left(f^{\lambda+\ell} g^{t-\lambda-1}\right): g^{t}=
\left(\left(  \frac{f^\ell}{g^\ell} \cJ(f^{\lambda} g^{t-\lambda}) \right) : g^{t}\right)\bigcap R=
\left( \frac{f^\ell}{g^\ell} \left( \cJ(f^{\lambda} g^{t-\lambda}): g^{t}\right)\right)\bigcap R.
$$
given by  Remark \ref{RemMixedMultMeromorphic}.
\end{proof}

We note that  Proposition \ref{PropSkoda}  gives a weaker version of Skoda's Theorem \cite[9.6.21]{LazBook2}.
In order to prove the next result, we  use some basic notions in  the theory of mixed multiplier ideals such as constancy regions and jumping walls that can be found in  \cite[\textsection 2.3]{AAD}.

\begin{theorem} \label{Thm: ratJN}
Let $R= \CC\{ x_1,\dots , x_n\}$ be the ring of holomorphic functions and $f,g\in R$ be  nonzero elements.
The set of jumping numbers of $f/g$  is a set of rational numbers with no accumulation points.
\end{theorem}
\begin{proof}
It suffices to show that the set of jumping numbers of  $f/g$ in $[0,t]$ is a finite set of rational numbers for any $t\in\ZZ_{\geq 0}$.
By Remark \ref{RemMixedMultMeromorphic}, the jumping numbers of  $\cJ (( \frac{f}{g})^\lambda)$ is a subset of the jumping numbers of $\cJ\left(f^{\lambda} g^{t-\lambda}\right)$. 
The last set is the intersection of the jumping walls of $\cJ(f^{\lambda_1} g^{\lambda_2})$ with the line $\lambda_1+\lambda_2=t$. 
  We conclude that the  set of jumping numbers of $f/g$  is a set of rational numbers with no accumulation points,  because the jumping walls are the boundaries of the constancy regions of $\cJ(f^{\lambda_1} g^{\lambda_2})$ and, by construction, they are defined by linear equations with rational coefficients and there are only finitely many in any bounded part of the positive orthant $\RR^2_{\geq 0}$.
\end{proof}
We state the following extension in our setting of the classical result of Ein, Lazarsfeld, Smith and Varolin \cite[Theorem 2.1]{ELSV} (see also \cite{Lic87}).

\begin{theorem} \label{Thm: JN_rootsBS}
Let $R= \CC\{ x_1,\dots , x_n\}$ be the ring of holomorphic functions and $f,g\in R$ be  nonzero elements.
Let $f,g$ be two nonzero elements of $R$ and take a jumping number $\lambda$ of $f/g$ such that $\lambda \in (1-\lct_0(g), 1]$. Then, $- \lambda$ is a root of the Bernstein-Sato polynomial of $f/g$.
\end{theorem}
\begin{proof} 
Suppose that $\lambda > 1- \lct_0(g)$ is a jumping number of $f/g$ and take $ 1- \lct_0(g)<\lambda' <\lambda $.
If we fix an arbitrary $c \in [\lambda', \lambda)$ then by the definition of jumping number there exist a function $h\in \cJ \left( \left( \frac{f}{g}\right)^c\right) $ such that $h\notin  \cJ \left( \left( \frac{f}{g}\right)^{\lambda-\varepsilon}\right)$, (we may assume without lost of generality that we are in the case $\lambda-\varepsilon$ of Definition \ref{Def:jumping_numbers}). Equivalently, there exist $r$ and $r'$, with $0 < r' <r$ and such that
\begin{equation}\label{Eq:Integrals}
	\int_{B_r(0)} \frac{|h|^2|g|^{2c}}{|f|^{2c}}<\infty \quad {\rm and } \quad 
	\int_{B_{r'}(0)} \frac{|h|^2|g|^{2(\lambda-\varepsilon)}}{|f|^{2(\lambda-\varepsilon)}}=\infty.
\end{equation}
We shall show that the first integral in Equation  (\ref{Eq:Integrals}) becomes unbounded when $c$ approaches $\lambda$. { To do so, note first that the second functional equation in Theorem \ref{Thm:GeneralizedFuncEq} implies the following analogue of  Equation  (\ref{Eq:IdentityConjugation})
\begin{equation*}
\delta(s)\overline{\delta}(s) \frac{|f|^{2(s+1)}}{|g|^{2(s+1)}} = b_{f/g}(s)^2  \frac{|f|^{2s}}{|g|^{2s}} .
\end{equation*}
} In particular, for  $s=-c$  and for any positive test function $\phi$ supported on $B_{r}(0)$ we have 

\begin{equation}\label{eq:aux1}
\begin{split}
	\int_{B_r(0)} |h|^2 \phi\   b_{f/g}(-c)^2  \frac{|g|^{2c}}{|f|^{2c}} &
	= \int_{B_r(0)} |h|^2  \phi\  \delta(-c)\overline{\delta}(-c) \frac{|f|^{2(-c+1)}}{|g|^{2(-c+1)}} \\
	&= \int_{B_r(0)}  \frac{|f|^{2(-c+1)}}{|g|^{2(-c+1)}}  \delta^*(-c)\overline{\delta}^*(-c)(|h|^2 \phi), 
\end{split}
\end{equation}
where  $\delta^*(s)$ and $\overline{\delta}^*(s)$ denote the respective adjoint operators. 
Since $c>\lambda' > 1-\lct_0(g)$, the proof of Lemma \ref{Lem:ConvergenceZetaF} implies that the right-hand side of Equation (\ref{eq:aux1}) is uniformly bounded  by a positive number  depending only on $h$ and $\phi$. If we now take $\phi$ as the characteristic function of the ball $B_{r'}(0)$, then
\[	\int_{B_r(0)} |h|^2 \phi\   b_{f/g}(-c)^2  \frac{|g|^{2c}}{|f|^{2c}} \geq 	\int_{B_r'(0)} |h|^2 \phi\   b_{f/g}(-c)^2  \frac{|g|^{2c}}{|f|^{2c}}=
b_{f/g}(-c)^2 \int_{B_{r'}(0)}   \frac{|h|^2|g|^{2c}}{|f|^{2c}}.
\]
This inequality and the boundedness of the right-hand side of Equation (\ref{eq:aux1}) imply that $b_{f/g}(-c)^2 \int_{B_{r'}(0)}   \frac{|h|^2|g|^{2c}}{|f|^{2c}}$ is bounded for any $c \in [\lambda', \lambda)$, but 
Equation  (\ref{Eq:Integrals}) shows that the last integral tends to infinity as $c$ tends to $\lambda$, implying that 
$b_{f/g}(-\lambda)=0$.

\end{proof}

\section{Multiplier ideals: algebraic construction} \label{Sec6}

In this section we define the algebraic version of multiplier ideals for meromorphic functions using log resolutions. 
Throughout this section we  use the notation introduced in Section \ref{Sec:log_resolution} concerning the numerical data associated to a log resolution $\pi : X \to U$ of the meromorphic germ  $f / g$.

\begin{definition}\label{DefMultiplier}
Let $R= \CC\{ x_1,\dots , x_n\}$ be the ring of holomorphic functions and $f,g\in R$ be  nonzero elements such that $f/g$ is not constant and $\lambda\in \RR_{\geq 0}$.
We define the $0$-multiplier ideal  of $f/g$ at $\lambda$ as the stalk at the origin of
\[\cJ \left( \left(\frac{f}{g}\right)^\lambda\right) =\pi_* \mathcal{O}_X ( -[ \lambda \cdot \tilde{F} ] + K_\pi), \]
where $[\cdot]$ denotes the integer part of a real number or $\mathbb{R}$-divisor. If no confusion arises we denote the stalk at the origin in the same way and thus $\cJ \left(\left( \frac{f}{g}\right)^\lambda\right) \subseteq R$.
\end{definition}

\begin{remark}\label{RemMult} $ $
\begin{enumerate}
\item Definition \ref{DefMultiplier} is independent of the choice of the log resolution $\pi$. This follows from analogous considerations to those in the classical case when $g=1$ \cite[Example 9.1.16, Theorem 9.2.18 \& Lemma 9.2.19]{LazBook2}.
\item Algebraic and analytic definitions of multiplier ideals for meromorphic functions coincide. The proof is an easy adaptation of the one for the classical case  \cite[Theorem 9.3.42]{LazBook2}. 
\item Given $h \in R$, the condition $h \in \cJ((\frac{f}{g})^\lambda)$ is equivalent to 
\[ {\rm ord}_{E_i} \pi^* h \geq [\lambda \cdot N_{f/g, i}] - k_i \quad \hbox{ for every } \, i \in I.\]
\item Since ${\rm ord}_{E_i} \pi^* h g^t = {\rm ord}_{E_i} \pi^* h + t N_{g, i}$, and $[\lambda \cdot N_{f/g, i}]+t N_{g, i}=[\lambda \cdot N_{f, i} + (t-\lambda) \cdot N_{g,i}]$ for all $t \in \NN$ such that $t \geq \lambda$, we have that the condition $h \in \cJ((\frac{f}{g})^\lambda)$ is equivalent to the conditions ${\rm ord}_{E_i} \pi^* h g^t \geq [\lambda \cdot N_{f, i} + (t-\lambda) \cdot N_{g,i}] + k_i$ for all $t \in \NN$ with $t \geq \lambda$ and all $i\in I$. This gives a different proof of Remark \ref{RemMixedMultMeromorphic}.
\item Given $h \in R$, we have ${\rm ord}_{E_i} \pi^* h \geq 0$ for all $i \in I$. Since $\cJ((\frac{f}{g})^\lambda) \subseteq R$, it follows  that \[ \cJ\left(\left(\frac{f}{g}\right)^\lambda\right) = \pi_* \mathcal{O}_X ( -[ \lambda \cdot \tilde{F}_0 ] + K_\pi). \]
\item One can associate $\infty$-multiplier ideals $\cJ^\infty((\frac{f}{g})^\lambda)$ to the meromorphic germ $f/g$ and the parameter $\lambda$ as follows $\cJ^\infty((\frac{f}{g})^\lambda)= \cJ((\frac{g}{f})^\lambda)$.
\item There are other versions of multiplier ideals for meromorphic functions. In fact, there is a version of multiplier ideals of ineffective divisors that gives a fractional ideal \cite[Remark 9.2.2]{LazBook2}.
\end{enumerate}
\end{remark}

The following properties of multiplier ideals  $ \cJ((\frac{f}{g})^\lambda)$ are analogous to the case of multiplier ideals associated to holomorphic germs $f$, that is, when $g=1$.

\begin{lemma}
There exists a discrete strictly increasing sequence of rational numbers
\[ \lambda_1 < \lambda_2 < \cdots \]
such that 
\[\cJ\left(\left(\frac{f}{g}\right)^{\lambda_{i+1}}\right) \subsetneq \cJ\left(\left(\frac{f}{g}\right)^c\right) = \cJ\left(\left(\frac{f}{g}\right)^{\lambda_i}\right)\]
for $c \in [\lambda_i, \lambda_{i+1})$, and all $i$. These rational numbers are called the $0$-jumping numbers of $f/g$. 
\end{lemma}

\begin{remark} $ $
\begin{itemize}
\item The candidate  $0$-jumping numbers of $f/g$ have the form
$\frac{k_i + 1+ \ell}{N_{f/g,i}}$ with  $\ell \in \mathbb{Z}_{\geq 0}$.
\item Let $\lambda \in \mathbb{Q}$ be a candidate jumping number of $f$ associated to the divisor $E_i$. Then $\lambda \cdot \frac{N_{f,i}}{N_{f/g,i}}$ is a candidate  $0$-jumping number of $f/g$. Notice that $\frac{N_{f,i}}{N_{f/g,i}} \geq 1$.
\end{itemize}
\end{remark}

\begin{lemma} \label{Lem: JN_integers} For any $\lambda \in \RR_{> 0}$ we have  $\cJ(f^\lambda) \subseteq \cJ((\frac{f}{g})^\lambda)$. In addition,   we have $\cJ((\frac{f}{g})^n)=(f^n)$ for any  $n \in \mathbb{Z}_{>0}$.
\end{lemma}
\begin{proof} 
The inclusion $\cJ(f^\lambda) \subseteq \cJ((\frac{f}{g})^\lambda)$ follows from the fact that $N_{f,i}/N_{f/g,i} \geq 1$. 
Therefore, we have that $(f^n) = \cJ(f^n) \subseteq \cJ((\frac{f}{g})^n)$. Moreover, for any $E_i$ in the strict transform of 
$f$ we have that $N_{f,i}/N_{f/g,i} =1$, and $k_i=0$. Hence, if $h \in \cJ((\frac{f}{g})^n)$, we have that 
$\ord_{E_i} \pi^* h \geq n \cdot N_{f,i}$. It follows that $f$ divides $h$ and hence $\cJ((\frac{f}{g})^n) \subseteq (f^n)$.
\end{proof}

\begin{remark} In general $\cJ((\frac{f}{g})^{\lambda+n}) \ne (f^n) \cdot \cJ((\frac{f}{g})^\lambda)$, as  Example \ref{ExMult} shows. As a consequence, the periodicity of jumping numbers fails. Alternatively, Skoda's Theorem for multiplier ideals of meromorphic germs  is weaker than the classic one. 
\end{remark}

\begin{example}\label{ExMult}
Let us consider $f=y^3+x^5$, $g_1=x$, $g_2=y$, and $f/g_i$ with $i=1,2$. The minimal resolution of $f g_i$ is obtained after 4 point blow-ups. Denote $E_i$ with $i=1,2,3,4$ the exceptional divisors, $E_5$ (resp. $E_6$) the strict transform of $f$ (resp. of $g_i$). Notice that $E_5$ intersects $E_4$, and $E_6$ intersects $E_1$ if $i=1$ or $E_2$ if $i=2$. Two additional point blow-ups, with exceptional divisors $E_7$ and $E_8$ are needed to construct the resolution of $f/g_i$. The following figures and tables give the corresponding dual resolution graphs and resolution data 
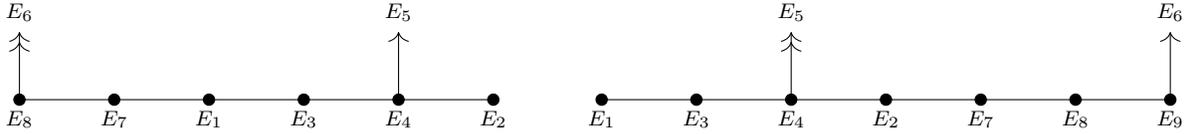
\begin{figure}[ht]
\begin{center}
\begin{tikzpicture}[scale=.9, vertice/.style={draw,circle,fill,minimum size=0.15cm,inner sep=0}]
\coordinate () at (0,0);
\coordinate (A) at (-9.8,2);
\coordinate (B) at (-8.4,2);
\coordinate (C) at (-7,2);
\coordinate (D) at (-5.6,2);
\coordinate (E) at (-4.2,2);
\coordinate (F) at (-2.8,2);
\coordinate (G) at (-9.8,3);
\coordinate (H) at (-4.2,3);

\node[vertice] at (A) {};
\node[below] at (A) {{\tiny $E_8$}};
\node[vertice] at (B) {};
\node[below] at (B) {{\tiny $E_7$}};
\node[vertice] at (C) {};
\node[below] at (C) {{\tiny $E_1$}};
\node[vertice] at (D) {};
\node[below] at (D) {{\tiny $E_3$}};
\node[vertice] at (E) {};
\node[below] at (E) {{\tiny $E_4$}};
\node[vertice] at (F) {};
\node[below] at (F) {{\tiny $E_2$}};
\node[above] at (G) {{\tiny $E_6$}};
\node[above] at (H) {{\tiny $E_5$}};

\draw (A)--(B)--(C)--(D)--(E)--(F);
\draw[-{[scale=1.5]>>}] (A)--(G) ;
\draw[-{[scale=1.5]>}] (E)--(H) ;

\coordinate (W) at (-1.2,2);
\coordinate (Z) at (0.2,2);
\coordinate (A1) at (1.6,2);
\coordinate (B1) at (3,2);
\coordinate (C1) at (4.4,2);
\coordinate (D1) at (5.8,2);
\coordinate (E1) at (7.2,2);

\coordinate (G1) at (1.6,3);
\coordinate (H1) at (7.2,3);

\node[vertice] at (W) {};
\node[below] at (W) {{\tiny $E_1$}};
\node[vertice] at (Z) {};
\node[below] at (Z) {{\tiny $E_3$}};
\node[vertice] at (A1) {};
\node[below] at (A1) {{\tiny $E_4$}};
\node[vertice] at (B1) {};
\node[below] at (B1) {{\tiny $E_2$}};
\node[vertice] at (C1) {};
\node[below] at (C1) {{\tiny $E_7$}};
\node[vertice] at (D1) {};
\node[below] at (D1) {{\tiny $E_8$}};
\node[vertice] at (E1) {};
\node[below] at (E1) {{\tiny $E_9$}};
\node[above] at (G1) {{\tiny $E_5$}};
\node[above] at (H1) {{\tiny $E_6$}};

\draw (W)--(Z)--(A1)--(B1)--(C1)--(D1)--(E1);
\draw[-{[scale=1.5]>>}] (A1)--(G1) ;
\draw[-{[scale=1.5]>}] (E1)--(H1) ;
\end{tikzpicture}
\end{center}
\end{figure}

\scalebox{0.9}{
$\begin{array}{c|cccccccc}
& E_1& E_2& E_3& E_4& E_5& E_6& E_7& E_8\\
\hline
N_{f, i}& 3& 5& 9& 15& 1& 0& 3& 3\\
N_{g_1,i}& 1& 1& 2& 3& 0& 1& 2& 3\\
k_i&1& 2& 4& 7& 0& 0& 2& 3\\
\end{array}
\qquad 
\begin{array}{c|ccccccccc}
& E_1& E_2& E_3& E_4& E_5& E_6& E_7& E_8& E_9\\
\hline
N_{f, i}& 3& 5& 9& 15& 1& 0& 5& 5& 5\\
N_{g_2,i}& 1& 2& 3& 5& 0& 1& 3& 4& 5\\
k_i&1& 2& 4& 7& 0& 0& 3& 4& 5\\
\end{array}$
}
\vspace{12pt}

The set of jumping numbers of $f/g_1$ is $\left\{\frac{8}{12}, \frac{11}{12}, 1, \frac{23}{12}\right\} \cup \ZZ_{\geq 2}$, the set of jumping numbers of $f/g_2$ is  $\left\{ \frac{8}{10}\right\} \cup \ZZ_{> 0}$  
and the multiplier ideals of $f/g_1$ and $f/g_2$ are

\scalebox{0.9}{
$\cJ\left(\left(\frac{f}{g_1}\right)^\lambda\right) = \begin{cases}
1 &   0 < \lambda < 8/12,\\
(x,y)  &  8/12  \leq \lambda < 11/12,\\
(x^2,y)  & 11/12  \leq \lambda < 1,\\
f &  1 \leq \lambda < 23/12,\\
(x, y)f &  23/12  \leq \lambda < 2\\
f^n &   n  \leq \lambda < n+1,\\
& \quad \hbox{ and } n \in \ZZ_{\geq 2}.
\end{cases}
 \quad \hbox{ and } \quad \cJ\left(\left(\frac{f}{g_2}\right)^\lambda\right) = \begin{cases}
1 &   0 < \lambda < 8/10,\\
(x,y)  &  8/10  \leq \lambda < 1,\\
f^n &   n  \leq \lambda < n+1,\\
& \quad \hbox{ and } n \in \ZZ_{\geq 1}.
\end{cases}$
}

For comparison's sake, recall that the set of jumping numbers of $f$ is $\left\{ \frac{8}{15}, \frac{11}{15}, \frac{13}{15}, \frac{14}{15} \right\} + \ZZ_{\geq 0}$, 
and the multiplier ideals of $f$ are

\[
\cJ(f^\lambda) = \begin{cases}
f^n &   n < \lambda < 8/15 + n, \hbox{ and } \lambda \ne 0,\\
(x,y) f^n  &  8/15 + n \leq \lambda < 11/15+n,\\
(x^2,y) f^n &  11/15 + n \leq \lambda < 13/15+n,\\
(x,y)^2 f^n &  13/15 + n \leq \lambda < 14/15+n,\\
(x^3, xy, y^2)f^n &   14/15 + n \leq \lambda < 1+n.
\end{cases}
\]

\end{example}

\section*{Acknowledgments}
We thank Nero Budur for pointing out the phenomenon described in Remark \ref{RmkNero},  Kevin Tucker for Remark \ref{RemMult}(6), and Ben Lichtin for pointing out a missing reference.  We also thank an anonymous referee for comments in a previous version of the paper that had a mistake.

\bibliographystyle{alpha}

\bibliography{References}

\end{document}